\RequirePackage{luatex85} 
\documentclass[11pt]{amsart}
\usepackage{amssymb}
\usepackage{latexsym}
\usepackage{graphicx}
\usepackage{subfigure}
\usepackage{tikz}
\usepackage[linktocpage=true]{hyperref}
\usepackage[left=3.8cm, right=3.8cm, top=3cm, bottom=3cm]{geometry}
\usepackage{color}
\usepackage{here}
\usepackage[backgroundcolor = lightgray,
            linecolor = lightgray,
            textsize = tiny]{todonotes}
\usepackage{stackengine}
\usepackage{mathrsfs}
\usepackage[all]{xy}
\usepackage[shortlabels]{enumitem}
\usepackage{scrextend}

\hypersetup{ colorlinks   = true, 
urlcolor  = blue, 
linkcolor    = blue, 
citecolor   = blue 
}

\newcommand{\dom}{\mathcal D}
\newcommand{\obs}{\mathcal O}

\def\beq{\begin{equation}}
\def\eeq{\end{equation}}

\newcommand{\billiards}{\mathbf{B}}

\newcommand{\Z}{{\mathbb Z}}
\newcommand{\R}{{\mathbb R}}

\newcommand{\T}{{\mathbb T}}

\newcommand{\N}{{\mathbb N}}

\theoremstyle{plain}
\newtheorem{theorem}{Theorem}[section]
\newtheorem{lemma}[theorem]{Lemma}

\newtheorem{prop}[theorem]{Proposition}
\newtheorem{corollary}[theorem]{Corollary}
\newtheorem{claim}[theorem]{Claim}

\newtheorem*{main thm}{Main Theorem}
\newtheorem*{main thm bis}{Main Theorem (alternate version)}

\newtheorem{theoalph}{Theorem}

\newtheorem{coralph}[theoalph]{Corollary}

\theoremstyle{definition}

\newtheorem{remark}[theorem]{Remark}
\newtheorem{defi}[theorem]{Definition}
\newtheorem{notation}[theorem]{Notation}

\begin{document}
\numberwithin{equation}{section}

\title{Smooth conjugacy classes of 3D Axiom A flows}

\author[Anna Florio]{Anna Florio$^1$}
\address{$^1$Sorbonne Universit\'e, Universit\'e de Paris, CNRS, Institut de Math\'ematiques de Jussieu-Paris Rive Gauche, F-75005 Paris, France \& CEREMADE, Univ. Paris Dauphine, Place du Mar\'echal de Lattre de Tassigny, F-75016 Paris, France}
\email{anna.florio@imj-prg.fr}

\author[Martin Leguil]{Martin Leguil$^2$}
\address{$^2$Laboratoire Ami\'enois de Math\'ematiques Fondamentales et Appliqu\'ees, CNRS-UMR 7352, Universit\'e de Picardie Jules Verne, 33 rue Saint Leu, 80039 Amiens cedex 1, France}
\email{martin.leguil@u-picardie.fr}

\maketitle

\begin{abstract}
	We show a rigidity result for $3$-dimensional contact \emph{Axiom A} flows: given two $\mathrm{3D}$ contact Axiom A flows $\Phi_1,\Phi_2$ whose restrictions $\Phi_1|_{\Lambda_1},\Phi_2|_{\Lambda_2}$ to basic sets $\Lambda_1,\Lambda_2$ are orbit equivalent, we prove that if periodic orbits in correspondence have the same length, then the conjugacy is as regular as the flows and respects the contact structure, extending a previous result due to Feldman-Ornstein \cite{FelOrn}. Some of the ideas are reminiscent of the work of Otal \cite{Otal}. As an application, we show that the billiard maps of two open dispersing billiards without eclipse and with the same marked length spectrum are \textit{smoothly} conjugated. 
\end{abstract}

\tableofcontents

\section{Introduction, statement of the results}

The concept of \emph{rigidity} arises in several ways in dynamics; one of them is the problem of knowing when two smooth systems which are topologically conjugated are actually \emph{smoothly} conjugated. It appears for instance in the framework of diffeomorphisms of the circle. In \cite{Arn} Arnold proved the first $\mathcal{C}^\omega$-linearization result. More precisely, he showed that an analytic diffeomorphism with Diophantine rotation number $\alpha$  and sufficiently close to the rotation $R_\alpha$ is analytically conjugated to $R_\alpha$. A global result in the  $\mathcal{C}^\infty$ category  is due to Herman, in \cite{Her}, where he also proved the optimality of the Diophantine condition in the smooth case; see also 
\cite{Yoccoz},  \cite{KhaSin} for related works. 

For low dimensional \emph{Anosov} systems, the question of rigidity has been investigated in many works, see for instance the series of papers by de la Llave, Marco and Moriy\'on  
\cite{MM,InvII,InvIII,InvIV}, \cite{LMM1}, and \cite{deLL}. While renormalization is one of the main tools behind the study of rigidity for circle diffeomorphisms, the approach for \emph{hyperbolic} systems is quite different. Indeed, for such systems, \emph{periodic orbits} are abundant, and each of them carries with itself an obstruction to smooth conjugacy, namely the associated eigenvalues of the differential. In the aforementioned works of de la Llave-Marco-Moriy\'on, it is shown that those obstructions are actually complete invariants for smooth conjugacy classes. The Anosov assumption can be relaxed, namely, we may consider systems where hyperbolicity is only observed on a subset of the phase space. In particular, when the \emph{non-wandering set} 
is hyperbolic, this leads to the notion of \emph{Axiom A} systems. In \cite{PRarxiv}, Pinto-Rand showed that Lipschitz conjugacy classes of hyperbolic basic sets on surfaces, which possess an invariant measure absolutely continuous with respect to Hausdorff measure, can be characterised in many ways, in particular, in terms of   eigenvalues at periodic points. Let us also mention the works \cite{PRrig} and \cite{BedFish}, where other rigidity results for hyperbolic sets have been obtained. In the context of expanding maps in any dimension, Gogolev and Rodriguez-Hertz \cite{RHG} have shown that, open and densely, smooth conjugacy classes are determined by the value of the Jacobian of the return maps at periodic points. 
 
 Let us now say a few words on rigidity questions in \emph{geometric}  frameworks. A natural setting is that of hyperbolic \emph{geodesic flows}. In this case, the general hope is that periodic data, in particular,  the \emph{length spectrum}, may be sufficient to characterize not only smooth conjugacy classes, but also to recover some \emph{geometric} information. The question of  \emph{spectral rigidity}  asks whether the (marked) length spectrum is sufficient to determine the metric up to isometry.  There exist various instances of this problem, both local and non-local. 
Guillemin-Kazhdan \cite{GuilKazh} have shown that compact negatively curved surfaces are spectrally rigid in the \emph{deformative} sense: a family $(g_s)_{s \in (0,1)}$ of isospectral negatively curved metrics is \emph{isometric}, that is, for each $s \in (0,1)$, there exists a diffeomorphism $\phi_s$ such that $g_s=\phi_s^* g_0$. Later, Paternain-Salo-Uhlmann \cite{PaternainSU} proved that any Anosov surface is spectrally rigid in the deformative sense. 
 Let us recall that for hyperbolic surfaces, periodic trajectories can be  naturally marked by  free homotopy classes.  The question of spectral rigidity for hyperbolic surfaces was addressed by
Otal \cite{Otal} and independently by Croke \cite{Croke}, who obtained the following \emph{global} result: two negatively curved metrics $g_0$ and $g_1$
 on a closed surface   with the same \emph{marked length spectrum} are isometric (see  also \cite{CrSh} for the multidimensional case). Recently, Guillarmou-Lefeuvre \cite{GuiLef} proved that in all dimensions, the marked length spectrum of a Riemannian manifold with Anosov geodesic flow and non-positive curvature locally determines the metric. See also the recent work \cite{RHG2} where a sharpened version of Otal and Croke's result was obtained. Other works have also investigated the case where the hyperbolic set is not the whole manifold. For instance, in \cite{Guill}, Guillarmou considers a smooth one-parameter family $(g_s)_{s \in (0,1)}$  of metrics on a smooth connected compact manifold with strictly convex boundary. When the metrics have no conjugate points, and the trapped set is a hyperbolic set for the geodesic flow, he proved that if all the metrics in the family are lens equivalent, then they are isometric. Following this work, Lefeuvre \cite{Lefeu} studied the X-ray transform on a smooth compact connected Riemannian manifold with hyperbolic trapped set. Other results in this direction have been recently obtained also by Chen, Erchenko and Gogolev in \cite{ChenErchGogol}.

Another setting where rigidity questions for the length spectrum have been investigated is the case of \textit{planar billiards}, in particular, for convex domains. One of the first results was obtained by Colin de Verdi\`ere \cite{DeVer}, who  established dynamical spectral rigidity for convex domains with analytic boundary and a $\Z_2\times \Z_2$-symmetry. De Simoi-Kaloshin-Wei \cite{dSKW} have proved dynamical spectral rigidity for $\Z_2$-symmetric strictly convex domains close to a circle; see also the recent work \cite{DeSAyub} by Ayub-De Simoi for ellipses of eccentricity smaller than $0.30$. Moreover, recently, for smoothly conjugate billiard maps of Birkhoff billiards, Kaloshin-Koudjinan \cite{KK} study rigidity in the form of Marvizi-Melrose invariants. Let us also recall that the question of spectral rigidity for convex billiards can be considered for other kinds of spectra as well: one of the most famous examples concerns the \textit{Dirichlet (or Neumann) spectrum}, which has been investigated at length, in particular, in a series of works by Zelditch \cite{Zel} and Hezari-Zelditch \cite{HeZ1,HeZ2,HeZ3,HeZ4}\dots, but also in many others. 

Yet, even more than convex billiards, hyperbolic billiards, in particular, \emph{dispersing billiards} are the most natural analogue of hyperbolic geodesic flows; indeed, although convex billiards may exhibit some hyperbolicity, for dispersing billiards, hyperbolicity is present on the whole phase space. The case of \emph{Sinai billiards} is very interesting, due to the abundance of periodic orbits; 
yet, the complicated structure of the set of periodic orbits as well as the presence of singularities make them hard to deal with. Many works have been dedicated to the study of \emph{open dispersing billiards} 
(see \cite{GARi,LoMar,Morita1991,Morita2004,Morita2007,Sto} for instance). Recall that the dynamics of open dispersing billiards is of \emph{Axiom A } type, and that their non-wandering set can be described symbolically (see \cite{Morita1991} for instance), which allows to define a marked length spectrum.  
The rigidity of scattering lengths and travelling times has been investigated in a series of works by Noakes, Stoyanov, and Petkov (see \cite{STNO,NoST,NoSto,StoLu} and also the book \cite{PetSto2}). In particular, \textit{lens rigidity} was established by Noakes-Stoyanov \cite{STNO,NoSto}: open dispersing billiards in $\R^d$, $d \geq 2$, are uniquely determined by the travelling times of billiard orbits and also by their  scattering length spectra. It is interesting to observe that as in the present work, the conjugacy between the billiard flows of two billiards with the same spectral data plays an important role: yet, Noakes-Stoyanov consider the conjugacy between the billiard flows \emph{outside the trapped set}, while here, we study the conjugacy precisely \emph{on the trapped set}; this is due to the fact that these works deal with different spectra. 
 
Let us conclude this introduction by mentioning recent rigidity results for hyperbolic billiards in terms of the \textit{length spectrum}. In \cite{KalChenZh}, Chen-Kaloshin-Zhang established the dynamical spectral rigidity of piecewise analytic Bunimovich stadia and squash type stadia. 
In \cite{DKL}, De Simoi-Kaloshin and the second author solved the question of marked length spectral determination for non-eclipsing open dispersing billiards with analytic boundary and two partial symmetries, 
under some mild non-degeneracy condition. Observe that in the $\mathcal{C}^k$ category, $k \in \N_{\geq 3} \cup \{+\infty\}$, the marked length spectrum is insufficient to fully determine the geometry of such tables; indeed,  periodic orbits are not dense in the whole phase space, so it is possible to deform the geometry of the arcs of the table which are not ``seen'' by the trapped set, i.e., which come from ``gaps'' of the projection on the table of the Cantor set on which we have information through periodic orbits.\\

In the present work,  we generalize the result of Feldman-Ornstein \cite{FelOrn} from contact Anosov flows on 3-manifolds to contact Axiom A flows on 3-manifolds. More precisely, equality of the length data allows us to upgrade an orbit equivalence to a 
flow conjugacy as regular as the flows, see Theorem \ref{thm princ dyn}. We apply this result to billiards exhibiting some hyperbolicity, and obtain a \textit{dynamical} rigidity result: for $k \geq 3$, we show that two $\mathcal{C}^k$ open dispersing billiards\footnote{Actually, the same result also holds for more general billiards, see Theorem \ref{theorem main billi}.} whose billiard maps are topologically conjugated on some horseshoe and have the same length data are actually \textit{smoothly} conjugated, in a \textit{canonical} way 
(see Theorem \ref{theorem main billi} and Corollary \ref{spectral rig cor}). In a previous version of this work, we were discussing geometric implications of the smoothness of the conjugacy, in connection with the question of spectral rigidity; yet, an error was found  in this part by Jacopo De Simoi. 

 

\subsection{Preliminaries}

Let $\Phi=(\Phi^t)_{t \in \R}$ be a continuous flow defined on a manifold $M$. For each point $x \in M$, we denote by $\mathcal{O}_\Phi(x):=\{\Phi^t(x)\}_{t \in \R}$ the $\Phi$-orbit of $x$. We denote by $\mathrm{Fix}(\Phi):=\{x \in M: \Phi^t(x)=x\text{ for all }t \in \R\}$ the set of \emph{fixed points} of $\Phi$, and we denote by $\mathrm{Per}(\Phi):=\{y \in M: \Phi^{T}(y)=y\text{ for some }T>0\}$ the set of \emph{periodic points} of $\Phi$; for any $x \in \mathrm{Per}(\Phi)$, we let $T_\Phi(x)=T_\Phi (\mathcal{O}_\Phi(x))>0$ be the prime period of $x$. Recall that the \emph{non-wandering set} $\Omega(\Phi)\subset M$ is the set of points $x\in M$ such that for any open set $U \ni x$, any $T_0>0$, there exists $T>T_0$ such that $\Phi^T (U) \cap U \neq \emptyset$. When $\Phi$ is a differentiable flow on some smooth manifold $M$, we denote by $X_\Phi(\cdot):=\frac{d}{dt}|_{t=0}\Phi(\cdot,t)$ its \emph{flow vector field}. 

In the following, given an integer $n \geq 1$, and $\beta \in (0,1)$, we say that a function $f$ is of class $\mathcal{C}^{n,\beta}$ if $f$ is $\mathcal{C}^{n}$, and its $n^{\mathrm{th}}$ derivative is $\beta$-H\"older continuous. 

\begin{defi}[Orbit equivalence]
	For $i=1,2$, let $\Phi_i=(\Phi_i^t)_{t \in \R}$ 
	be a flow defined on a manifold $M_i$, and let $\Lambda_i\subset M_i$ be a $\Phi_i$-invariant subset. 
	We say that the flows $\Phi_1,\Phi_2$ are \emph{orbit equivalent} on $\Lambda_1,\Lambda_2$ if there exists a  homeomorphism $\Psi \colon \Lambda_1 \to \Lambda_2$ such that for some continuous function $\theta \colon \Lambda_1 \times \R \to \R$, we have for each $x \in \Lambda_1$:
	\begin{itemize}
		\item $\theta (x,0)=0$, and $\theta (x,\cdot)$ is an increasing $\mathcal{C}^{1,\beta}$ homeomorphism of $\R$, for some $\beta\in (0,1)$;   
		\item $\Psi \circ \Phi_1^t(x)=\Phi_2^{\theta (x,t)} \circ \Psi(x)$, for all $t \in \R$.
	\end{itemize}  
In other words, $\Psi$ sends $\Phi_1$-orbits to $\Phi_2$-orbits: 
$$
\Psi (\mathcal{O}_{\Phi_1}(x))=\mathcal{O}_{\Phi_2} (\Psi(x)),\quad \text{for all } x \in \Lambda_1.
$$
Recall that $\Psi$ is automatically $\mathcal{C}^{\delta}$ for some  $\delta\in (0,1)$,  if $\Lambda_1,\Lambda_2$ are compact hyperbolic sets (see Katok-Hasselblatt \cite[Theorem 19.1.2]{HaKa}). 

Moreover, we say that $\Psi$ is 
\emph{iso-length-spectral} if 
$$
T_{\Phi_1}(x)=T_{\Phi_2}(\Psi(x)),\quad \forall\, x\in \mathrm{Per}(\Phi_1)\cap \Lambda_1,
$$
i.e., the flows $\Phi_1,\Phi_2$ have the same periodic length data.

If $M_1,M_2$ are smooth, and $\Phi_1,\Phi_2$ are differentiable flows, we abbreviate as $X_i:=X_{\Phi_i}$ the flow vector field, for $i =1,2$, and we say that $\Psi$ is \emph{differentiable along 
	$\Phi_1$-orbits} (in $\Lambda_1$) 
if the Lie derivative $$
\Lambda_1\ni x \mapsto L_{X_1}  \Psi(x):=\lim_{t \to 0} \frac{1}{t}\big(\Psi\circ \Phi_1^t(x)-\Psi(x)\big)\in \R X_2\circ \Psi(x)
$$ 
is a well-defined continuous function. 
\end{defi}

\begin{defi}[Adapted contact form]\label{adapted contact form}
	Given a smooth (connected) $3$-manifold $M$, recall that a \emph{contact form} is a smooth differential one-form that satisfies the \emph{non-integrability condition} $\alpha\wedge d \alpha \neq 0$; without loss of generality, we may assume that $\alpha\wedge d \alpha > 0$.    
	
	Let $k \geq 2$, and let $\Phi=(\Phi^t)_{t \in \R}$ be a $\mathcal{C}^k$ Axiom A flow defined on a smooth $3$-manifold $M$. Given a basic set $\Lambda\subset M$ for $\Phi$, we say that a contact form $\alpha$ is \emph{adapted to $\Lambda$} if it satisfies the following Reeb conditions:
	\begin{enumerate}[(a)]
		\item\label{condition un} $\imath_{X}\alpha|_{\Lambda}\equiv 1$;
		\item\label{condition deux} $X|_{\mathcal{W}_\Phi^{cs}(\Lambda)} \in \ker d\alpha|_{\mathcal{W}_\Phi^{cs}(\Lambda)}$ and $X|_{\mathcal{W}_\Phi^{cu}(\Lambda)} \in \ker d\alpha|_{\mathcal{W}_\Phi^{cu}(\Lambda)}$.
	\end{enumerate}
\end{defi}

In the following, we fix a  $\mathcal{C}^\infty$ smooth Riemannian manifold $M$, and we consider a  $\mathcal{C}^2$  flow  $\Phi=(\Phi^t)_{t \in \R}$ on $M$.  

\begin{defi}[Hyperbolic set]
	A $\Phi$-invariant compact subset $\Lambda \subset M\setminus \mathrm{Fix}(\Phi)$ is called a \textit{(uniformly) hyperbolic} set (for $\Phi$) if there exists a $D\Phi$-invariant splitting 
	$$
	T_x M=E^s(x) \oplus \R X(x) \oplus E^u(x),\qquad \forall\, x \in \Lambda,
	$$
	where the \emph{(strong) stable bundle} $E_\Phi^s$, resp. the \emph{(strong) unstable bundle} $E_\Phi^u$ is uniformly contracted, resp. expanded, i.e., there exist $C>0$, $\lambda\in (0,1)$ such that
	\begin{align*}
	\| D\Phi^t(x) \cdot v\| &\leq C \lambda^t \|v\|,\qquad \forall\, x \in \Lambda,\, \forall\, v\in E_\Phi^s(x),\, \forall\, t \geq 0, \\
	\| D\Phi^{-t}(x) \cdot v\| &\leq C \lambda^t \|v\|,\qquad \forall\, x \in \Lambda,\,  \forall\, v \in E_\Phi^u(x),\, \forall\, t \geq 0. 
	\end{align*}
	We also denote by $E_\Phi^{cs}$, resp. $E_\Phi^{cu}$, the \emph{weak stable bundle} $E_\Phi^{cs}:=E_\Phi^s \oplus \R X$, resp. the \emph{weak unstable bundle}  $E_\Phi^{cu}:=\R X\oplus E_\Phi^u$. 
\end{defi}

Let us recall the definition of an \emph{Axiom A} flow:
\begin{defi}[Axiom A flow]\label{def axiom a}
	A flow $\Phi \colon M\times\R \to M$ is called \emph{Axiom A} if the non-wandering set $\Omega(\Phi) \subset M$ can be written as a disjoint union $\Omega(\Phi)=\Lambda \cup F$, where $\Lambda$ is a closed hyperbolic set such that periodic orbits are dense in $\Lambda$,  and $F\subset \mathrm{Fix}(\Phi)$ is a finite union of hyperbolic fixed points.  
\end{defi}

\begin{defi}[Lamination]
	Let $n \geq 1$, $\beta\in (0,1)$. A $\mathcal{C}^{n,\beta}$-lamination of a set $\Lambda \subset M$ is a disjoint collection of $\mathcal{C}^{n,\beta}$ submanifolds of a given same dimension, which vary continuously in the $\mathcal{C}^{n,\beta}$-topology, and whose union contains the set $\Lambda$. 
\end{defi}

Let $\Phi\colon M\times\R \to M$ be an Axiom A flow with a decomposition $\Omega(\Phi)=\Lambda \cup F$ as in Definition \ref{def axiom a}. The stable bundle $E_\Phi^s$, resp. the unstable bundle $E_\Phi^u$, over $\Lambda$ integrates to a continuous lamination $\mathcal{W}^s_\Phi$, resp. $\mathcal{W}_\Phi^{u}$, called the \emph{(strong) stable lamination}, resp. the \emph{(strong) unstable lamination}.  Similarly, $E_\Phi^{cs}$, resp. $E_\Phi^{cu}$ integrates to a continuous lamination $\mathcal{W}_\Phi^{cs}$, resp. $\mathcal{W}_\Phi^{cu}$, called the \emph{weak stable lamination}, resp. the \emph{weak unstable lamination}. For each point $x \in \Lambda$, a local orbit segment in $\mathcal{O}_\Phi(x)$ containing $x$ will also be denoted as $\mathcal{W}_{\Phi,\mathrm{loc}}^c(x)=\mathcal{W}_{\Phi,\mathrm{loc}}^{cs}(x)\cap \mathcal{W}_{\Phi,\mathrm{loc}}^{cu}(x)$. 
Each of these laminations is invariant under the dynamics, i.e., $\Phi^t(\mathcal{W}_\Phi^*(x))=\mathcal{W}_\Phi^*(\Phi^t(x))$, for all $x \in M$ and $*=s,u,c,cs,cu$. For each subset $S \subset \Lambda$, we also denote  $\mathcal{W}^*_\Phi(S):=\cup_{x \in S} \mathcal{W}^*_\Phi(x)$, for $*=s,u,c,cs,cu$. 

Besides, we have $\Lambda=\Lambda_1\cup \cdots \cup \Lambda_m$ for some integer $m \geq 1$, where for each $i \in \{1,\dots,m\}$,  $\Lambda_i$ is a hyperbolic set such that  $\Phi|_{\Lambda_i}$ is transitive, and $\Lambda_i=\cap_{t \in \R} \Phi^t (U_i)$ for some open set $U_i \supset \Lambda_i$. The set $\Lambda_i$ is called a \emph{basic set} of $\Phi$. 

\begin{remark}\label{remarque reg}
	In general, the stable/unstable distributions $E_{\mathcal{F}}^{s/u}$ at a hyperbolic invariant set $\Lambda$ of some diffeomorphism $\mathcal{F}$ are only H\"older continuous, but according to Pinto-Rand \cite{PinRan}, when the stable, resp. unstable leaves are 
	one-dimensional, and $\Lambda$ has local product structure, then the stable holonomies, resp. unstable holonomies are of class $\mathcal{C}^{1,\beta}$, $\beta \in (0,1)$. In our case, both distributions are one-dimensional, so the holonomies will be $\mathcal{C}^{1,\beta}$, for some $\beta \in (0,1)$.
\end{remark}

Let us recall the following version of the extension theorem due to Whitney \cite{Whitney}. It legitimates the notion of differentiability in Whitney sense.

\begin{theorem}
	Fix an integer $k\geq 1$. Let $A \subset \R^n$ be a closed subset, $n \geq 1$, and let $f_0,\dots,f_k \colon A \to \R$ be 
	continuous functions such that  for some $\beta\in (0,1)$, it holds 
	\begin{equation}\label{f zero}
	f_0(y)-f_0(x)=\sum_{j=1}^k \dfrac{f_j(x)}{j!}(y-x)^j +O(|y-x|^{k+\beta}),\quad \forall\, x,y \in A.
	\end{equation}
	Then, there exists a $\mathcal{C}^{k,\beta}$ function $f \colon \R^n \to \R$ such that $f|_A=f_0|_A$, $f^{(j)}|_{A}=f_j|_A$ for $j=1,\dots, k$, and $f|_{\R^n\setminus A}$ is $\mathcal{C}^\omega$. A function $f_0\colon A \to \R$ which satisfies \eqref{f zero} for some functions $f_1,\dots, f_k\colon A \to \R$ is said to be $\mathcal{C}^{k,\beta}$ in Whitney sense. 
\end{theorem}

\subsection{Dynamical spectral rigidity of contact Axiom A flows}

Our main dynamical result is the following. 
\begin{theoalph}[Length spectral rigidity on basic sets]\label{thm princ dyn}
	Fix $k \geq 2$. For $i=1,2$, let $\Phi_i=(\Phi_i^t)_{t \in \R}$ be a $\mathcal{C}^k$ Axiom A flow defined on a $3$-manifold $M_i$. Let $\Lambda_i$ be a basic set for $\Phi_i$, and assume that there exists a smooth contact form $\alpha_i$ on $M_i$ that is adapted to $\Lambda_i$. If there exists an orbit equivalence $\Psi_0\colon \Lambda_1 \to  \Lambda_2$ between $\Phi_1|_{\Lambda_1}$ and $\Phi_2|_{\Lambda_2}$ that is differentiable along  $\Phi_1$-orbits and iso-length-spectral, 
	then
	 \begin{enumerate}
	 	\item $\Phi_1|_{\Lambda_1}$, $\Phi_2|_{\Lambda_2}$ are $\mathcal{C}^{k}$-conjugate; 
	 	more precisely,  there exists a H\"older continuous homeomorphism $\Psi\colon \Lambda_1 \to  \Lambda_2$ that is $\mathcal{C}^{k}$ in Whitney sense, such that
	 	$$
	 	\Psi \circ \Phi_1^t(x)=\Phi_2^{t} \circ \Psi(x),\quad  \text{for all } (x,t) \in \Lambda_1\times \R;
	 	$$
	 	\item $\Psi$ preserves the contact form, i.e., $\Psi^* \alpha_2|_{\Lambda_1}=\alpha_1|_{\Lambda_1}$. 
	 \end{enumerate}
\end{theoalph}

In other terms, iso-length-spectral orbit equivalence classes between basic sets of $\mathcal{C}^k$ Axiom A flows with an adapted contact form are in one-to-one correspondence with $\mathcal{C}^{k}$ flow conjugacy classes 
between these basic sets, where the conjugacy preserves the contact form. Besides, it will be clear from the proof that the $\mathcal{C}^k$-regularity is actually needed on $\Lambda_i$ (in Whitney sense).  

\begin{remark}\label{remark Fisher Hassl}
	Let $\Phi_1,\Phi_2$, and let $\Lambda_1$, $\Lambda_2$ be as in Theorem \ref{thm princ dyn}. The flow conjugacy $\Psi\colon \Lambda_1 \to  \Lambda_2$ between $\Phi_1|_{\Lambda_1}$ and $\Phi_2|_{\Lambda_2}$ given by Theorem \ref{thm princ dyn} is essentially unique. Indeed, for any other flow conjugacy $\widetilde{\Psi}\colon \Lambda_1 \to  \Lambda_2$, it holds 
	$$
	(\Psi^{-1}\circ \widetilde{\Psi} )\circ \Phi_1^t=\Phi_1^t \circ (\Psi^{-1}\circ \widetilde{\Psi})\quad \text{on }\Lambda_1,
	$$ 
	that is, $\Psi^{-1}\circ\widetilde{\Psi}$ is in the diffeomorphism centralizer of $\Phi_1|_{\Lambda_1}$. By \cite[Theorem 1.4]{BFH}, the centralizer is trivial, hence $\widetilde{\Psi}=\Psi\circ\Phi_1^T$, for some $T\in\R$. 
	In Subsection \ref{section symmetry ies}, we explain that in some cases (when the system has a time-reversal symmetry) there is a natural way to choose $T$ so as to make the conjugacy \textit{canonical}. 
\end{remark}

 Since the Hausdorff dimension is preserved by Lipschitz continuous homeomorphisms, and since the stable/unstable Hausdorff dimensions are constant on $\Lambda$ (see for instance \cite{Pesinlecochino}), we deduce from Theorem \ref{thm princ dyn} the following result:

\begin{coralph}
	Let $\Phi=(\Phi^t)_{t \in \R}$ be a $\mathcal{C}^k$ Axiom A flow defined on a smooth $3$-manifold $M$, $k \geq 2$. Let $\Lambda$ be a basic set for $\Phi$ with an adapted smooth contact form $\alpha$. Then, the Hausdorff dimensions $\mathrm{dim}_H(\Lambda)$, $\delta^{(s)}(\Lambda)$, $\delta^{(u)}(\Lambda)$ are invariant under iso-length-spectral orbit equivalences, where for $*=s,u$, we let $\delta^{(*)}(\Lambda)=\delta^{(*)}:=\dim_{H}(\Lambda\cap \mathcal{W}_{\Phi}^*(x))$, for any $x \in \Lambda$. 
\end{coralph}

\subsection{Open dispersing billiards}\label{sec prelim}
We consider a billiard table $\dom = \R^{2}\setminus \bigcup^{\ell}_{i = 1}\obs_{i}$ obtained by removing from the plane $\ell\geq 3$ \emph{obstacles} $\obs_1,\dots,\obs_\ell$, each 
$\obs_{i}$ being a convex domain with $\mathcal{C}^k$ boundary $\partial\obs_{i}$, for some $k \geq 3$, such that $\overline{\obs_1},\dots,\overline{\obs_\ell}$ are pairwise disjoint. For each $i \in \{1,\dots,\ell\}$, we let $|\partial\obs_{i}|$ be the
corresponding perimeter, and parametrize
each $\partial\obs_{i}$ counterclockwisely  in arc-length by some map
$\Upsilon_i \in \mathcal{C}^k( \T_i, \R^2)$, $s\mapsto \Upsilon_i(s)$, where $\T_i:=\R/(|\partial\obs_{i}|\Z)$.  
The set of all such billiard tables will be denoted by $\billiards$, and for each $\ell \geq 3$, we let $\billiards(\ell)\subset \billiards$ be the subset of tables with $\ell$ obstacles. 

Let 
$\dom = \R^{2}\setminus \bigcup^{\ell}_{i = 1}\obs_{i} \in
\billiards$, for some $\ell \geq 3$. We denote the collision space by
\begin{align*}
\mathcal{M} &:= \bigcup_i \mathcal{M}_i, &
\mathcal{M}_i &:=\{(q,v),\ q \in \partial\obs_{i},\ v\in \R^2,\ \|v\|=1,\ \langle v,n\rangle\geq 0\},
\end{align*}
where $n$ is the unit normal
vector to $\partial\obs_{i}$ pointing outside $\obs_i$. For each
$x=(q,v) \in \mathcal{M}$, we have $q=\Upsilon_i(s)$, for some $i\in \{1,\dots,\ell\}$ and some arc-length
parameter $s \in \mathbb{T}_i$; we let $\varphi\in [-\frac{\pi}{2},\frac{\pi}{2}]$
be the oriented angle between $n$ and $v$, and set $r:=\sin \varphi$. Therefore,  each $\mathcal{M}_i$ can be seen as
a cylinder $\T_i \times [-1,1]$ endowed with
coordinates $(s,r)$. In the following, given a point $x=(s,r) \in \mathcal{M}$, we let
$\Upsilon(s):=q$ be the associated point of $\partial \mathcal{D}$. 

For each pair $(s_1,r_1), (s_2,r_2) \in \mathcal{M}$, we denote by
\begin{equation}\label{def hsspremie}
h(s_1,s_2):=\|\Upsilon(s_1)-\Upsilon(s_2)\|
\end{equation}
the Euclidean length of the segment connecting the associated points 
of the table.

Let $\mathfrak{M} :=\{(q,v) \in \mathcal{D} \times \mathbb{S}^1\}/\sim$ be the quotient of $\mathcal{D} \times \mathbb{S}^1$ by the relation $\sim$:
$$
(q_1,v_1)\sim (q_2,v_2)\quad \Longleftrightarrow \quad q_1=q_2 \in \partial \mathcal{D}\text{ and } v_2=\mathcal{R}_{q_1}(v_1),
$$ 
where $\mathcal{R}_{q_1}$ is the reflection in $\R^2$ with respect to the tangent line $T_{q_1}\partial \mathcal{D}$. An element of $\mathfrak{M}$ will be denoted as $[(q,v)]$. 
In the following, we identify a point $[(q,v)]\in \mathfrak{M}$, $q\in \partial \mathcal{D}$, with the corresponding element $(q,v)\in \mathcal{M}$. 
Let $\Phi=(\Phi^t)_{t \in \R}$ be the associated billiard flow on $\mathfrak{M}$. We can describe this flow with coordinates $(x,y,\omega)$, where $(x,y)\in \R^2$ are the Cartesian coordinates of some point $q \in \mathcal{D}$ on the table and $\omega \in [0,2\pi)$ denotes the couterclockwise angle between the positive $x$ axis and the velocity vector $v$. For each $x \in \mathcal{M}$, we let $\tau(x)\in \R_+\cup \{+\infty\}$ be the first return time of the $\Phi$-orbit of $x$ to $\mathcal{M}$, and denote by
$$
\mathcal{F}=\mathcal{F}(\mathcal{D}) \colon \mathcal{M}\cap \{\tau \neq +\infty\} \to \mathcal{M},\quad x \mapsto \Phi^{\tau(x)}(x)
$$
the associated billiard map, which we see as a map 
$
\mathcal{F}\colon (s,r)\mapsto (s',r')
$, with $s'=s'(s,r)$ and $r'=r'(s,r)$. 

\begin{figure}[h]
	\includegraphics[scale=0.7, trim=0 1.3cm 0 1cm, clip]{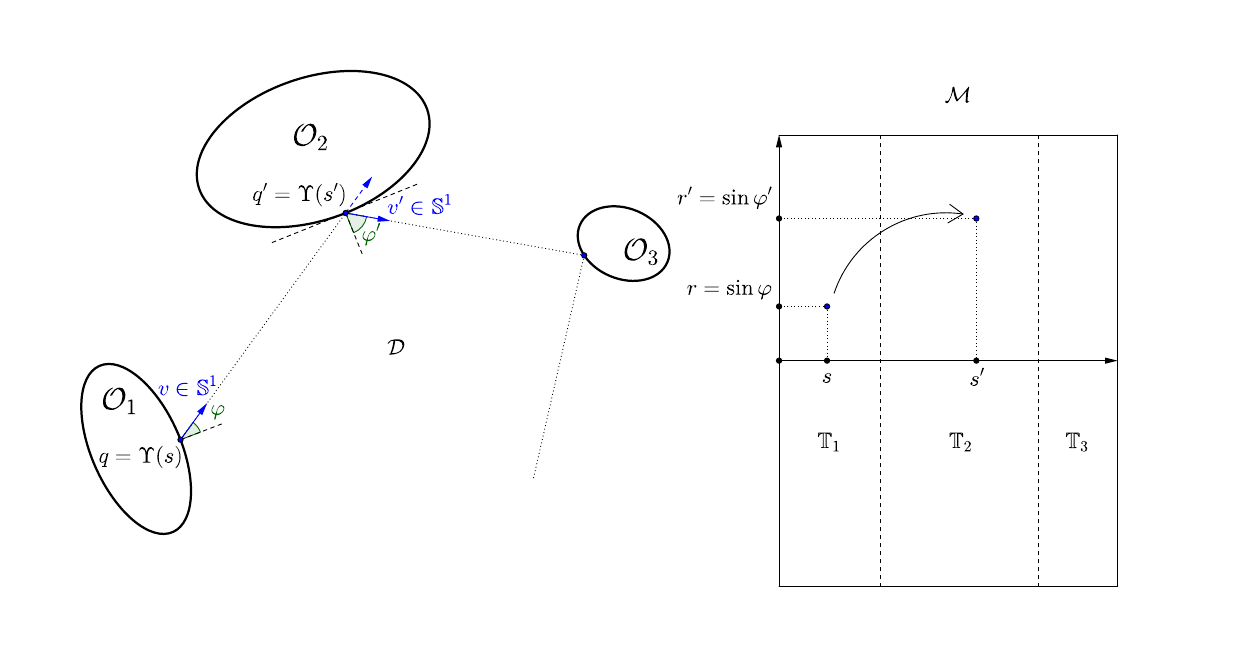}
	\caption{An open dispersing billiard and its phase space.}
\end{figure}

For any point $x=(s,r)\in \mathcal{M}$ with a well-defined image $(s',r')=\mathcal{F}(s,r)$, recall that $h=h(s,s')$ is the distance between the two points of collision. Note that $h(s,s')=h(s,s'(s,r))=\tau(s,r)$ is the first return time of $(s,r)\in \mathcal{M}$ to $\mathcal{M}$. Let $\mathcal{K}:=\mathcal{K}(s)$, $\mathcal{K}':=\mathcal{K}(s')$ be the respective curvatures, and set $\nu=\nu(r):=\sqrt{1-r^2}$,  $\nu':=\nu(r')=\sqrt{1-(r')^2}$. 
By the formulas in Chernov-Markarian \cite{ChernovMarkarian}, the differential of the billiard map is
\begin{equation}\label{magtrice diff billi}
D\mathcal{F}(s,r)=-\begin{bmatrix}
\frac{1}{ \nu'}(h \mathcal{K}+\nu) & \frac{h}{\nu \nu'}\\
h \mathcal{K}\mathcal{K}'+\mathcal{K} \nu'+\mathcal{K}'\nu & \frac{1}{\nu}(h \mathcal{K}'+\nu')
\end{bmatrix}.
\end{equation}
The map $\mathcal{F}$ is exact symplectic for the Liouville form $\lambda=-rds$:
\begin{equation}\label{gen fun billllii}
\mathcal{F}^* \lambda - \lambda = d\tau.
\end{equation}
Fix a lift $\widetilde{\mathcal{F}}$ of $\mathcal{F}$ to $\R \times [-1,1]$. We let $|\partial \mathcal{D}|:=|\partial \obs_1|+\dots+|\partial \obs_m|$ be the total perimeter, and extend the definition of $h$ by letting $h(s+p|\partial \mathcal{D}|,s'+q|\partial \mathcal{D}|)=h(s,s')$, for any $p,q \in \Z$. Then, $h$ is a \emph{generating function}\footnote{In the following, we will also refer to the function $\tau=\tau(s,r)$ as a generating function.} for the dynamics of $\widetilde{\mathcal{F}}$ (or $\mathcal{F}$):
\begin{equation*}
\left\{
\begin{array}{lcl}
r &=& \frac{\partial h(s,s')}{\partial s},\\
r' &=& -\frac{\partial h(s,s')}{\partial s'}.
\end{array}
\right.
\end{equation*}
Observe that $\mathcal{F}$ is a negative twist map, i.e., $\frac{\partial s'}{\partial r}(s,r)<0$, and that $-\frac{\partial^2 h}{\partial s \partial s'}(s,s')>0$. Let us also recall that the time-reversal involution $\mathcal{I}\colon (s,r)\mapsto (s,-r)$ conjugates the billiard map with its inverse, i.e., $\mathcal{F}\circ \mathcal{I} = \mathcal{I}\circ \mathcal{F}^{-1}$. 

Due to the strict convexity of the obstacles, the dynamics is of Axiom A type (see  \cite{Morita2004,Morita2007} or \cite[Subsection 2.1]{Sto} for more details). In connection with Remark \ref{remarque reg}, let us also recall that several works have been dedicated to the smoothness of stable/unstable laminations of open dispersing billiards 
(see Morita \cite{Morita2004} and Stoyanov \cite{Sto}). 
Besides, if the non-wandering set 
$$
\Omega(\mathcal{F}):= \bigcap_{j \in \Z} \mathcal{F}^{j}(\mathcal{M})
$$  
has no tangential collisions, then it 
is a hyperbolic set; moreover, we have $\Omega(\mathcal{F})=\Lambda \cup F$,  $\Lambda \cap F=\emptyset$, where $F$ is a finite union of periodic points, and  
$\Lambda$ can be written as a disjoint union $\Lambda=\Lambda_1\cup\dots \cup \Lambda_m$, $m \geq 1$, each $\Lambda_i$ being a horseshoe such that $\mathcal{F}|_{\Lambda_i}$ is conjugated to a non-trivial subshift of finite type. In the following, for each point $x\in\Omega(\mathcal{F})$, we denote by $\mathcal{W}_\mathcal{F}^s(x)$, resp. $\mathcal{W}_\mathcal{F}^u(x)$, its stable, resp. unstable manifold for the map $\mathcal{F}$.   The non-wandering set $\Omega(\Phi)$ of the billiard flow $\Phi$ is the set of all points in the orbit of some $x \in \Omega(\mathcal{F})$. Similarly, when speaking about a basic set for $\Phi$ in the following, we mean the union of orbits of all the points in a set $\Lambda_i$ as above, for some $i\in\{1,\dots,m\}$. 
Let us define the quotient set $\Lambda_i^\tau:=\{(s,r,t)\in \Lambda_i\times \R : 0 \leq t \leq \tau(s,r)\}/\approx$, where 
$$
((s,r),\tau(s,r))\approx (\mathcal{F}(s,r),0).
$$
We can identify $\Lambda_i^\tau$ with the set $\{(s,r,t)\in \Lambda_i\times \R: 0 \leq t < \tau(s,r)\}$, and define the projection $\Pi \colon \Lambda_i^\tau \to \partial \mathcal{D}$ as\footnote{By a slight abuse of notation, we will also denote by $\Pi \colon \Lambda_i \to \partial \mathcal{D}$ the projection $(s,r)\mapsto s$.}
\begin{equation}\label{projection obst}
\Pi(s,r,t):=s \simeq \Upsilon(s).
\end{equation}
The billiard flow $\Phi$ restricted to the orbits of points in $\Lambda_i$ is defined at all times and can be seen as a \emph{special flow} induced by the vertical vector field $X=\frac{\partial}{\partial t}=(0,0,1)$ on $\Lambda_i^{\tau}$. 
 Actually, the $(x,y,\omega)$-coordinates introduced above are slightly more convenient, as they also allow to describe points which are not in $\Omega(\Phi)$: for any point $(s,r,t)\in \Lambda_i^\tau$, with $r=\sin \varphi \in (-1,1)$, $t \in [0,\tau(s,r))$, we let $U(s,r,t):=(x,y,\omega)\in \mathfrak{M}$ be the corresponding $(x,y,\omega)$-coordinates, with $x=x(s,r,t)$, $y=y(s,r,t)$, and 
 \begin{gather*}
 \omega=\omega(s,r)=\angle \big(R_{-\frac \pi 2+\varphi}\big(\Upsilon'(s)\big),(1,0)\big)=\angle \big(R_{-\frac \pi 2+\arcsin r}\big(\Upsilon'(s)\big),(1,0)\big),\\
 \big(x(s,r,t),y(s,r,t) \big)=\Upsilon(s)+t(\cos \omega,\sin \omega),
 \end{gather*}
 where $\Upsilon(s)$ is the associated point of $\partial \mathcal{D}$, and for $\theta \in \R$, $R_{\theta}$ is the rotation of angle $\theta$. The map $\Upsilon$ is $\mathcal{C}^k$, hence the change of coordinates $U$ is of class $\mathcal{C}^{k-1}$.
 
\begin{claim}
	The contact form $\alpha=\lambda+dt$ is adapted to $\Lambda_i^\tau$ (recall Definition \ref{adapted contact form}).
\end{claim} 

\begin{proof}
Let us verify that $\imath_X \alpha=1$ and $\imath_X d\alpha=0$. Indeed, for any $(s,r,t)=((s,r),t)\in \mathcal{M} \times \R$, we have 
$$
\alpha(s,r,t) \big(X(s,r,t)\big)=(\lambda(s,r)+dt)\frac{\partial}{\partial t}=1,
$$  
and
$$
d\alpha(s,r,t)(X(s,r,t))=d\lambda(s,r)\frac{\partial}{\partial t}=0.
$$
Besides, for $W\colon(s,r,t)\mapsto (\mathcal{F}(s,r),t-\tau(s,r))$, we have 
\begin{align*}
W^* \alpha (s,r,t) &= \alpha \circ W (s,r,t) = \alpha(\mathcal{F}(s,r),t- \tau(s,r))\\
&=\lambda(\mathcal{F}(s,r))+d(t-\tau(s,r))=\mathcal{F}^*\lambda(s,r)+dt-d\tau(s,r)\\
&=\lambda(s,r)+d\tau(s,r)+dt-d\tau(s,r)=\alpha(s,r,t).
\end{align*}
Therefore, $\alpha$ descends to an adapted contact form on $\Lambda_i^\tau$. 
\end{proof}

Let us also recall how the contact structure looks like in $(x,y,\omega)$-coordinates.  
For each point $X=(x,y,\omega) \in \mathfrak{M}$, we let
\begin{align*}
T_X\mathfrak{M}\supset T_X^0\mathfrak{M} &:=\ker\big( -\sin \omega dx+\cos \omega dy\big) \cap \ker \big(d\omega\big),\\
T_X\mathfrak{M}\supset T_X^\perp\mathfrak{M} &:=\ker \big(\cos \omega dx+\sin \omega dy\big).
\end{align*}
The one-dimensional subbundle $T^0\mathfrak{M}\subset T\mathfrak{M}$ and the two-dimensional subbundle $T^\perp\mathfrak{M}\subset T\mathfrak{M}$ are $D \Phi$-invariant. 
More precisely, for any $t \in \R$, the differential $D\Phi^t$ acts on $T^0\mathfrak{M} \oplus T^\perp\mathfrak{M}$ as follows (see Chernov-Markarian \cite{ChernovMarkarian} for more details):
$$
D\Phi^t(X)=\begin{bmatrix}
1 & 0 \\
0 & D^\perp\Phi^t(X)
\end{bmatrix},\quad \forall\, X \in \mathfrak{M}. 
$$
In particular, $\cos \omega dx+\sin \omega dy$ is the contact form in $(x,y,\omega)$-coordinates, and $T^\perp\mathfrak{M}$ is the associated contact distribution. 

\begin{defi}
	Let $\mathcal{D}_1,\mathcal{D}_2$ be two billiards with $\mathcal{C}^k$ boundaries, for some $k\geq 3$, and let $\Phi_1,\Phi_2$ be the associated billiard flows. If there exist two basic sets $\Lambda_1^{\tau_1}\subset \Omega(\Phi_1)$, $\Lambda_2^{\tau_2}\subset \Omega(\Phi_2)$, and an iso-length-spectral orbit equivalence between $\Phi_1|_{\Lambda_1^{\tau_1}}$ and $\Phi_2|_{\Lambda_2^{\tau_2}}$, we will simply say that $\mathcal{D}_1,\mathcal{D}_2$ are \textit{iso-length-spectral} on $\Lambda_1^{\tau_1},\Lambda_2^{\tau_2}$.
\end{defi}

\begin{theoalph}[Smooth conjugacy of billiard maps of isospectral hyperbolic billiards]\label{theorem main billi}
Let $\mathcal{D}_1,\mathcal{D}_2$ be two billiards with $\mathcal{C}^k$ boundaries, for some $k\geq 3$, and let $\Phi_1,\Phi_2$ be the associated billiard flows.  
Let us consider a basic set $\Lambda_i^{\tau_i}$ for $\Phi_i$, $i=1,2$, and let $\Lambda_i$ be the \emph{horseshoe}\footnote{Let us recall that a \textit{horseshoe} for a diffeomorphism $f$ is a transitive, locally maximal hyperbolic set that is totally disconnected and not finite.}  obtained by projecting $\Lambda_i^{\tau_i}$ onto the first two coordinates $(s_i,r_i)$. 
If $\mathcal{D}_1,\mathcal{D}_2$ are iso-length-spectral on $\Lambda_1^{\tau_1}$, $\Lambda_2^{\tau_2}$, then 
there exists a map $\widetilde{\Psi}\colon (s_1,r_1,t_1)\mapsto (s_2,r_2,t_2)$ 
which conjugates $\Phi_1,\Phi_2$ on $\Lambda_1^{\tau_1},\Lambda_2^{\tau_2}$ respectively. The map $\widetilde{\Psi}$ induces a conjugacy $\Psi\colon \Lambda_1 \to \Lambda_2$ between the respective billiard maps $\mathcal{F}_1|_{\Lambda_1},\mathcal{F}_2|_{\Lambda_2}$ which is $\mathcal{C}^{k-1}$ in Whitney sense, and such that $\Psi^* (ds_2\wedge dr_2)=ds_1\wedge dr_1$ on $\Lambda_1$. 

Moreover, the respective  generating functions $\tau_1,\tau_2$ of $\mathcal{F}_1,\mathcal{F}_2$ satisfy
\begin{equation}\label{cohomology functions gen}
\tau_2 \circ \Psi - \tau_1 = \chi\circ \mathcal{F}_1 - \chi\quad \text{on }\Lambda_1,
\end{equation}
for some function $\chi\colon \Lambda_1 \to \R$ which is $\mathcal{C}^{k-1}$ in Whitney sense, such that 
\begin{equation}\label{liouville im}
\Psi^* \lambda_2-\lambda_1 = d\chi\quad \text{on }\Lambda_1,\quad\text{where }\lambda_i=-r_i ds_i,\ i=1,2.
\end{equation}
Let $\mathcal{I}_i \colon (s_i,r_i)\mapsto (s_i,-r_i)$, $i=1,2$, be the respective time-reversal involutions. If $\Psi^{-1}\circ \mathcal{I}_2 \circ \Psi\circ \mathcal{I}_1|_{\Lambda_1}$ fixes  $\mathcal{F}_1$-orbits, i.e., $\mathcal{I}_2\circ \Psi(x_1)$ and $\Psi \circ\mathcal{I}_1(x_1)$ are in the same $\mathcal{F}_2$-orbit, for all $x_1\in \Lambda_1$, and if there exists a $2$-periodic point $x_1 \in \Lambda_1$ or a point  $x_1 \in \Lambda_1\cap \{r_1=0\}$ whose orbit is dense in $\Lambda_1$ and such that $\mathcal{F}_2^m\circ \Psi(x_1) \in \{r_2=0\}$, for some $m \in \Z$, then $\Psi$, resp. $\chi$, can be chosen in a unique way such that 
\begin{equation}\label{image temp rev}
\Psi \circ \mathcal{I}_1|_{\Lambda_1} = \mathcal{I}_2 \circ \Psi|_{\Lambda_1},\qquad \text{resp. }\chi\circ \mathcal{I}_1|_{\Lambda_1}=-\chi|_{\Lambda_1}.
\end{equation} 
\end{theoalph}

The proof of Theorem \ref{theorem main billi} is given in Section \ref{section rig bill}.
 
\begin{remark}
Theorem \ref{theorem main billi} applies naturally to open dispersing billiards, as those exhibit uniformly hyperbolic dynamics. Yet, even in the case of convex billiards, generically, hyperbolic dynamics arises from \emph{Aubry-Mather} periodic orbits with transverse heteroclinic intersections (see for instance \cite{ArnaudMC,HKS} for more details). Thus, our result may also be applied to the associated horseshoes.
\end{remark}

\begin{remark}\label{stable action spectrum}
The coboundary $\chi$ in Theorem \ref{theorem main billi} can be interpreted as the difference between stable (or unstable) actions for the billiard maps $\mathcal{F}_1,\mathcal{F}_2$. Indeed, fix a $2$-periodic point $p_1\in \Lambda_1$, and let 
$p_2:=\Psi(p_1)\in \Lambda_2$. Let us consider a point $x_1  \in \Lambda_1$ in the stable manifold $\mathcal{W}_{\mathcal{F}_1}^s(p_1)$ of $p_1$, and let $x_2:=\Psi(x_1)\in\mathcal{W}_{\mathcal{F}_2}^s(p_2)\cap \Lambda_2$. 
For $i=1,2$, we define the \textit{stable action} of $x_i$ as the sum of the following convergent series:
$$
\mathcal{A}_{p_i,\mathcal{F}_i}^s(x_i)=\mathcal{A}_i^s(x_i):=\sum_{k=0}^{+\infty} \big(\tau_i \circ \mathcal{F}_i^k(x_i)-\tau_i \circ \mathcal{F}_i^k(p_i)\big).
$$ 
Since the two billiards have the same periodic length data, and since $p_1,p_2$ are $2$-periodic, we have $\tau_1 \circ \mathcal{F}_1^k(p_1)=\tau_2 \circ \mathcal{F}_2^k(p_2)$, for each $k \in \Z$.  Observe that 
$
\lim_{k \to+\infty}\chi \circ \mathcal{F}_1^k(x_1)=\chi(p_1)=0
$ (as $\chi(p_1)=\chi\circ \mathcal{I}_1(p_1)=-\chi(p_1)$, by \eqref{image temp rev}). 
By \eqref{cohomology functions gen}, we thus conclude that
\begin{equation*}
\mathcal{A}_1^s(x_1)-\mathcal{A}_2^s(x_2)=\sum_{k=0}^{+\infty} \big(\tau_1 \circ \mathcal{F}_1^k(x_1)-\tau_2 \circ \mathcal{F}_2^k(x_2)\big)
=\chi(x_1)-\lim_{k \to+\infty}\chi \circ \mathcal{F}_1^k(x_1)=\chi(x_1),
\end{equation*}
i.e., $\chi(x_1)$ is the difference between the stable actions $\mathcal{A}_1^s(x_1)$ and $\mathcal{A}_2^s(\Psi(x_1))$.
\end{remark}

\subsection{Open dispersing billiards without eclipse}

We now discuss the following important example (see for instance \cite{Morita1991,DKL} for more details). 
Fix an integer $\ell \geq 3$. We let $\billiards_{ne}(\ell)\subset \billiards(\ell)$ be the set of all billiards
$\dom = \R^{2}\setminus \bigcup^{\ell}_{i = 1}\obs_{i}\in \billiards(\ell)$ which satisfy the following\\

\noindent \textsc{Non-eclipse condition:} The convex hull  of any two obstacles
is disjoint from any other obstacle.\\

\begin{figure}[h]
	\includegraphics[scale=0.5, trim=0 1.2cm 0 1.1cm, clip]{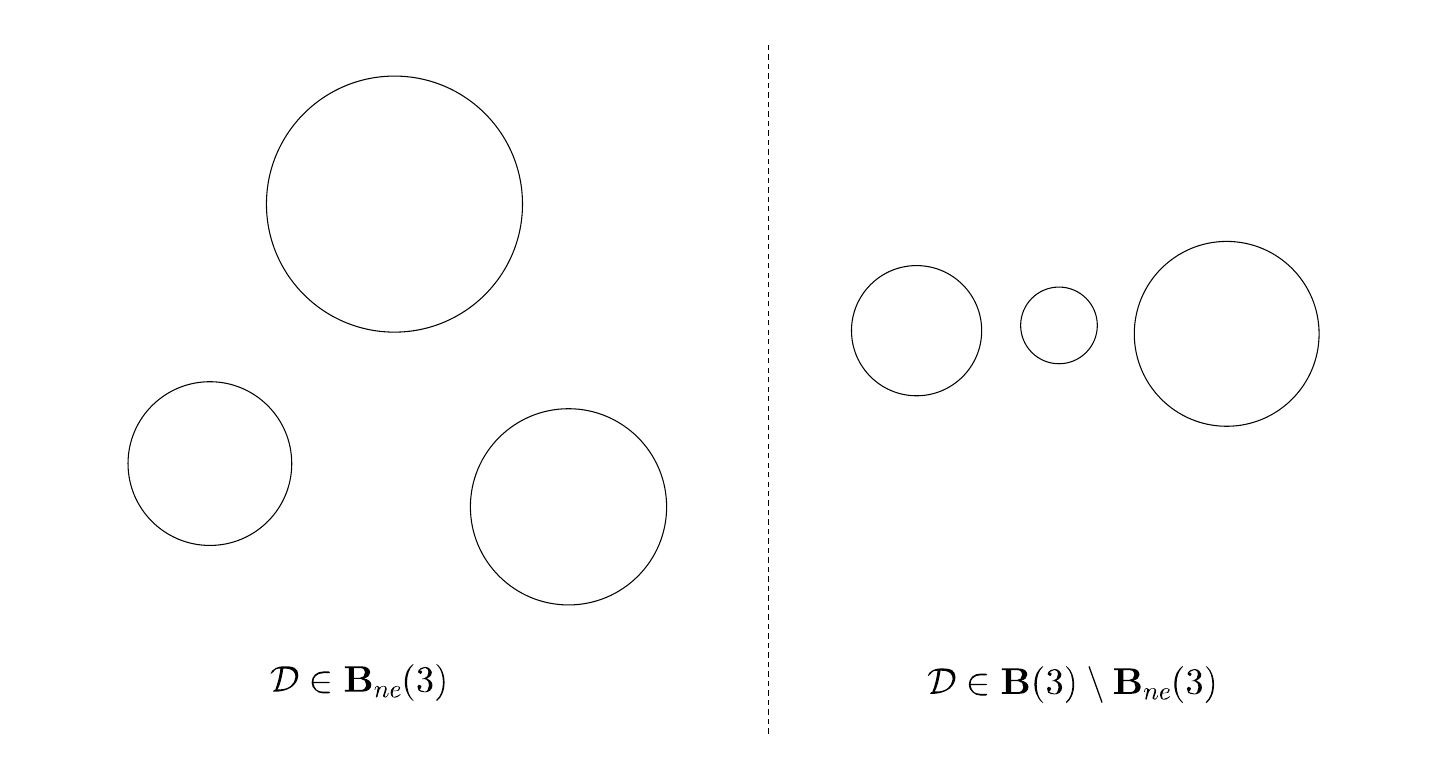}
\end{figure}

Let $\mathcal{F}$, resp. $\Phi$ be the associated billiard map, resp. billiard flow. 
The non-wandering set $\Omega(\mathcal{F})$ is reduced to a single basic set $\Lambda$. Moreover,  $\mathcal{F}|_\Lambda$ is conjugated by some H\"older homeomorphism to the subshift of finite type associated with the transition matrix $(1-\delta_{i,j})_{1 \leq i,j \leq \ell}$, where $\delta_{i,j}=1$, when $i=j$, and $\delta_{i,j}=0$ otherwise, when $i\neq j$. 
In other words, any \textit{admissible} word $\varsigma\in \text{Adm}_{\infty}$, i.e., such that $\varsigma=(\varsigma_j)_{j}\in \{1,\dots,\ell\}^\Z$ with
$\varsigma_{j+1}\neq \varsigma_j$, for all $j \in \Z$, can be realized
by an orbit, and by hyperbolicity of the dynamics, this orbit is
unique. We denote by $x(\varsigma)\in \Omega(\mathcal{F})$ the point with symbolic coding $\varsigma$. Let $\text{Adm} \subset \cup_{j \geq 2} \{1,\dots,\ell\}^j$ be the set of all finite words $\sigma=\sigma_1\dots\sigma_j$, $j \geq 2$, such that $\sigma^\infty:=\cdots\sigma\sigma\sigma\cdots \in \text{Adm}_{\infty}$. It is the set of symbolic codings of periodic orbits. In particular, we may thus define the \emph{marked length spectrum} $\mathcal{MLS}(\mathcal{D})$ as the map 
$$
\mathcal{MLS}(\mathcal{D})\colon \text{Adm} \to \R,\quad \sigma \mapsto \mathcal{L}(\sigma),
$$
where $\mathcal{L}(\sigma)=T_\Phi(x(\sigma^\infty))$ is the perimeter of the periodic orbit encoded by $\sigma$. 

For any billiards $\dom_1,\dom_2\in \billiards_{ne}(\ell)$ with respective billiard maps $\mathcal{F}_1,\mathcal{F}_2$, the restrictions $\mathcal{F}_1|_{\Omega(\mathcal{F}_1)}$, $\mathcal{F}_2|_{\Omega(\mathcal{F}_2)}$ are topologically conjugated in a canonical way, by sending a point $x_1\in \Omega(\mathcal{F}_1)$ to the point $x_2\in \Omega(\mathcal{F}_2)$ with the same coding. The billiard flows $\Phi_1,\Phi_2$ are thus orbit equivalent through some H\"older continuous orbit equivalence.  

\begin{coralph}
	\label{spectral rig cor}
		Fix $\ell \geq 3$, and let $\mathcal{D}_1,\mathcal{D}_2 \in \billiards_{ne}(\ell)$ with $\mathcal{C}^k$ boundaries, for some $k\geq 3$. If $\mathcal{D}_1,\mathcal{D}_2$ have the same marked length spectrum, then the respective billiard maps $\mathcal{F}_1,\mathcal{F}_2$ are conjugated on $\Omega_1:=\Omega(\mathcal{F}_1)$, $\Omega_2:=\Omega(\mathcal{F}_2)$ by a map $\Psi\colon \Omega_1 \to \Omega_2$ that is $\mathcal{C}^{k-1}$ in Whitney sense, such that $\Psi^* (ds_2\wedge dr_2)=ds_1\wedge dr_1$ and $\Psi \circ \mathcal{I}_1 = \mathcal{I}_2 \circ \Psi$ on $\Omega_1$, where $\mathcal{I}_i \colon (s_i,r_i)\mapsto (s_i,-r_i)$, for $i=1,2$, is the time-reversal involution. Moreover, \eqref{cohomology functions gen}-\eqref{liouville im}-\eqref{image temp rev} in Theorem \ref{theorem main billi} hold for some coboundary $\chi\colon \Lambda_1 \to \R$ which is $\mathcal{C}^{k-1}$ in Whitney sense.
\end{coralph}


\noindent{\textbf{Acknowledgements:} We thank P\'eter B\'alint, Sylvain Crovisier,  Jacques F\'ejoz, Andrey Gogolev, Colin Guillarmou, Umberto L. Hryniewicz, Thibault Lefeuvre, Jean-Pierre Marco, Federico Rodriguez Hertz, Disheng Xu for their encouragement and several useful discussions. We are especially grateful to Marie-Claude Arnaud, Jacopo De Simoi, Livio Flaminio, Vadim Kaloshin, and Ke Zhang for many conversations about this work, and for having pointed out a mistake in a previous version of the paper.

\section{Smooth conjugacy classes for $\mathrm{3D}$ Axiom A flows on basic sets}

\subsection{Synchronization of the flows using periodic data}
 
Let us start by recalling the fact that an orbit equivalence between two hyperbolic flows can be upgraded to a flow conjugacy as long as the lengths of associated periodic orbits coincide. 

\begin{prop}\label{upgrading orbit eq}
	Let $k \geq 2$, and let $\Phi_1=(\Phi_1^t)_{t \in \R}$, resp. $\Phi_2=(\Phi_2^t)_{t \in \R}$ be a $\mathcal{C}^k$ Axiom A flow defined on a smooth manifold $M_1$, resp. $M_2$, and let $\Lambda_1$, resp. $\Lambda_2$ be a basic set for $\Phi_1$, resp. $\Phi_2$. Assume that there exists an orbit equivalence $\Psi_0\colon \Lambda_1 \to  \Lambda_2$ differentiable along  $\Phi_1$-orbits, and that
	\begin{equation}\label{same length}
	T_{\Phi_1}(x)=T_{\Phi_2}(\Psi(x)),\quad \text{
	for each } x\in \mathrm{Per}(\Phi_1)\cap \Lambda_1.
	\end{equation}
	Then the flows $\Phi_1,\Phi_2$ are topologically conjugate, i.e., there exists a homeomorphism $\Psi\colon \Lambda_1 \to  \Lambda_2$ such that 
		$$
		\Psi \circ \Phi_1^t(x)=\Phi_2^{t} \circ \Psi(x),\quad  \text{for all } (x,t) \in \Lambda_1\times \R. 
		$$
\end{prop}

\begin{proof}
The proof is classical but we recall it here for completeness. 

We fix an orbit equivalence $\Psi_0\colon \Lambda_1 \to  \Lambda_2$ that is differentiable along $\Phi_1$-orbits. 
Let $X_1,X_2$ be the respective flow vector fields of $\Phi_1,\Phi_2$, and let  $L_{X_1} \Psi_0$ be the Lie derivative of $\Psi_0$ along $\Phi_1$. As $\Psi_0$ sends $\Phi_1$-orbits to $\Phi_2$-orbits, it holds 
$$
L_{X_1} \Psi_0(x)=v_{\Psi_0}(x) X_2( \Psi_0(x)),\quad \text{for all }x \in \Lambda_1,
$$
for some function $v_{\Psi_0} \colon \Lambda_1\to \R$ which measures the ``speed" of $\Psi_0$ along the flow direction. Observe that $v_{\Psi_0}(x)=\frac{d}{dt}|_{t=0}\theta (x,t)$.

By \eqref{same length}, for each $x\in \mathrm{Per}(\Phi_1)\cap\Lambda_1$ we have 
$$
\int_0^{T_{\Phi_1}(x)}dt=T_{\Phi_1}(x)=T_{\Phi_2}(\Psi_0(x))=\int_0^{T_{\Phi_1}(x)}\frac{d}{ds}|_{s=0}\theta (\Phi_1^t(x),s)dt=\int_0^{T_{\Phi_1}(x)}v_{\Psi_0}(\Phi_1^t(x))dt,
$$
hence
$$
\frac{1}{T_{\Phi_1}(x)}\int_{0}^{T_{\Phi_1}(x)} \big(v_{\Psi_0}(\Phi_1^t(x)) -1\big) dt=0,\quad \text{
	for each } x\in \mathrm{Per}(\Phi_1)\cap \Lambda_1.
$$ 
We deduce from Livsic's theorem (see \cite[Subsection 19.2]{HaKa}) that there exists a continuous function $u \colon \Lambda_1 \to \R$ differentiable along $\Phi_1$-orbits such that  $v_{\Psi_0}-1=L_{X_{1}} u$.  Let us set $\Psi\colon x \mapsto \Phi_2^{- u(x)} \circ \Psi_0 (x)$. Given any $x \in \Lambda_1$, we compute
\begin{align*}
v_{\Psi}(x)X_2(\Psi(x))&=L_{X_{1}}\big(\Phi_2^{- u(x)} \circ \Psi_0 \big)(x)\\
&=\lim_{t\rightarrow0}\frac{1}{t}\Big(\Phi_2^{\theta (x,t)-u(\Phi_1^t(x))}\circ\Psi_0(x)-\Phi_2^{-u(x)}\circ\Psi_0(x)\Big)\\
&=X_2(\Phi_2^{-u(x)}\circ\Psi_0(x))\, \lim_{t\rightarrow0}\frac{1}{t}\Big(\theta (x,t)-u(\Phi_1^t(x))+u(x)\Big)\\
&=\big(v_{\Psi_0} (x) - L_{X_{1}} u (x)\big) X_2\big(\Phi_2^{- u(x)} \circ \Psi_0  (x)\big)=X_2(\Psi (x)),
\end{align*}
i.e., $v_{\Psi}\equiv 1$ on $\Lambda_1$.

As a result, the homeomorphism $\Psi$ is a flow conjugacy between $\Phi_1$ and $\Phi_2$ on $\Lambda_1$:
$$
\Psi \circ \Phi_1^t(x)=\Phi_2^{t} \circ \Psi(x),\quad  \text{for all } (x,t) \in \Lambda_1\times \R. 
$$
\end{proof}

\subsection{Markov families for Axiom A flows on basic sets}\label{subsection markov}

In this part, we recall some classical facts about Markov families for Axiom A flows on basic sets, following the presentation given in \cite{Chernov}. 

	Let $k \geq 2$, and let $\Phi=(\Phi^t)_{t \in \R}$ be a $\mathcal{C}^k$ Axiom A flow defined on a smooth manifold $M$. 

\begin{defi}[Rectangle, proper family] 
	A closed subset $R \subset M$ is called a \emph{rectangle} if there is a small closed codimension one smooth disk $D \subset M$ transverse to the flow $\Phi$ such that $R \subset D$, and for any $x,y \in R$, the point 
	$$
	[x,y]_R:=D \cap \mathcal{W}_{\Phi,\mathrm{loc}}^{cs}(x)\cap \mathcal{W}_{\Phi,\mathrm{loc}}^{cu}(y)
	$$
	exists and also belongs to $R$. A  rectangle $R$ is called \emph{proper} if $R=\overline{\mathrm{int}(R)}$ in the topology of $D$. For any rectangle $R$ and any $x \in R$, we let
	$$
	\mathcal{W}_R^s(x):= R \cap \mathcal{W}_{\Phi,\mathrm{loc}}^{cs}(x),\quad \mathcal{W}_R^u(x):= R \cap \mathcal{W}_{\Phi,\mathrm{loc}}^{cu}(x).
	$$
	
	A finite collection of proper rectangles $\mathcal{R}=\{R_1,\dots,R_m\}$, $m \geq 1$, is called a \emph{proper family of size} $\varepsilon>0$ if:
	\begin{enumerate}
		\item $M=\{\Phi^{t}(\mathcal{S}):t \in [-\varepsilon,0]\}$, where $\mathcal{S}:=R_1\cup \dots \cup R_m$;
		\item $\mathrm{diam}(D_i)< \varepsilon$, for each $i=1,\dots,m$, where $D_i\supset R_i$ is a disk   as above;
		\item for any $i\neq j$, $D_i\cap \{\Phi^{t}(D_j):t \in [0,\varepsilon]\}=\emptyset$ or $D_j\cap \{\Phi^{t}(D_i):t \in [0,\varepsilon]\}=\emptyset$. 
	\end{enumerate}
	The set $\mathcal{S}$ is called a \emph{cross-section} of the flow $\Phi$.
\end{defi}

\begin{notation}
	Let $\mathcal{R}=\{R_1,\dots,R_m\}$ be a proper family with $m \geq 1$ elements. 
	
	The cross-section $\mathcal{S}:=R_1\cup \dots \cup R_m$ is associated with a Poincar\'e map $\mathcal{F}\colon \mathcal{S}\to \mathcal{S}$, where for any $x \in \mathcal{S}$, we let $\mathcal{F}(x):=\Phi^{\tau_\mathcal{S}(x)} (x)$,  the function $\tau_\mathcal{S}\colon \mathcal{S}\to \R_+$ being the first return time on $\mathcal{S}$, i.e., $\tau_\mathcal{S}(x):=\inf\{ t> 0 : \Phi^t(x)\in\mathcal{S} \}>0$, for all $x \in \mathcal{S}$. 
	
	
	Besides, for $*=s,u$ and $x \in R_i$, $i\in \{1,\dots,m\}$, we also let $\mathcal{W}_{\mathcal{F}}^*(x):=\mathcal{W}_{R_i}^*(x)$. 
\end{notation}

\begin{defi}[Markov family]
	Given some small $\varepsilon>0$, and some integer $m\geq 1$, a proper family $\mathcal{R}=\{R_1,\dots,R_m\}$ of size $\varepsilon$, with Poincar\'e map $\mathcal{F}$, is called a \emph{Markov family} if it satisfies the following Markov property:
	for any $x \in \mathrm{int} (R_i)  \cap \mathcal{F}^{-1}(\mathrm{int} (R_j))\cap \mathcal{F}(\mathrm{int} (R_k))$, with $i,j,k\in \{1,\dots,m\}$, it holds 
	$$
	\mathcal{W}_{R_i}^s(x)\subset \overline{\mathcal{F}^{-1}(\mathcal{W}_{R_j}^s(\mathcal{F}(x)))}\quad \text{and}\quad \mathcal{W}_{R_i}^u(x)\subset\overline{\mathcal{F}(\mathcal{W}_{R_k}^u(\mathcal{F}^{-1}(x)))}.
	$$
\end{defi}

\begin{theorem}[see Theorem 4.2 in \cite{Chernov}]\label{Chernov Markov}
	The restriction of an Axiom A flow to any basic set has a Markov family of arbitrary small size.
\end{theorem}

\begin{figure}[h]
	\includegraphics[scale=0.5, trim=0 0.5cm 0 0.5cm, clip]{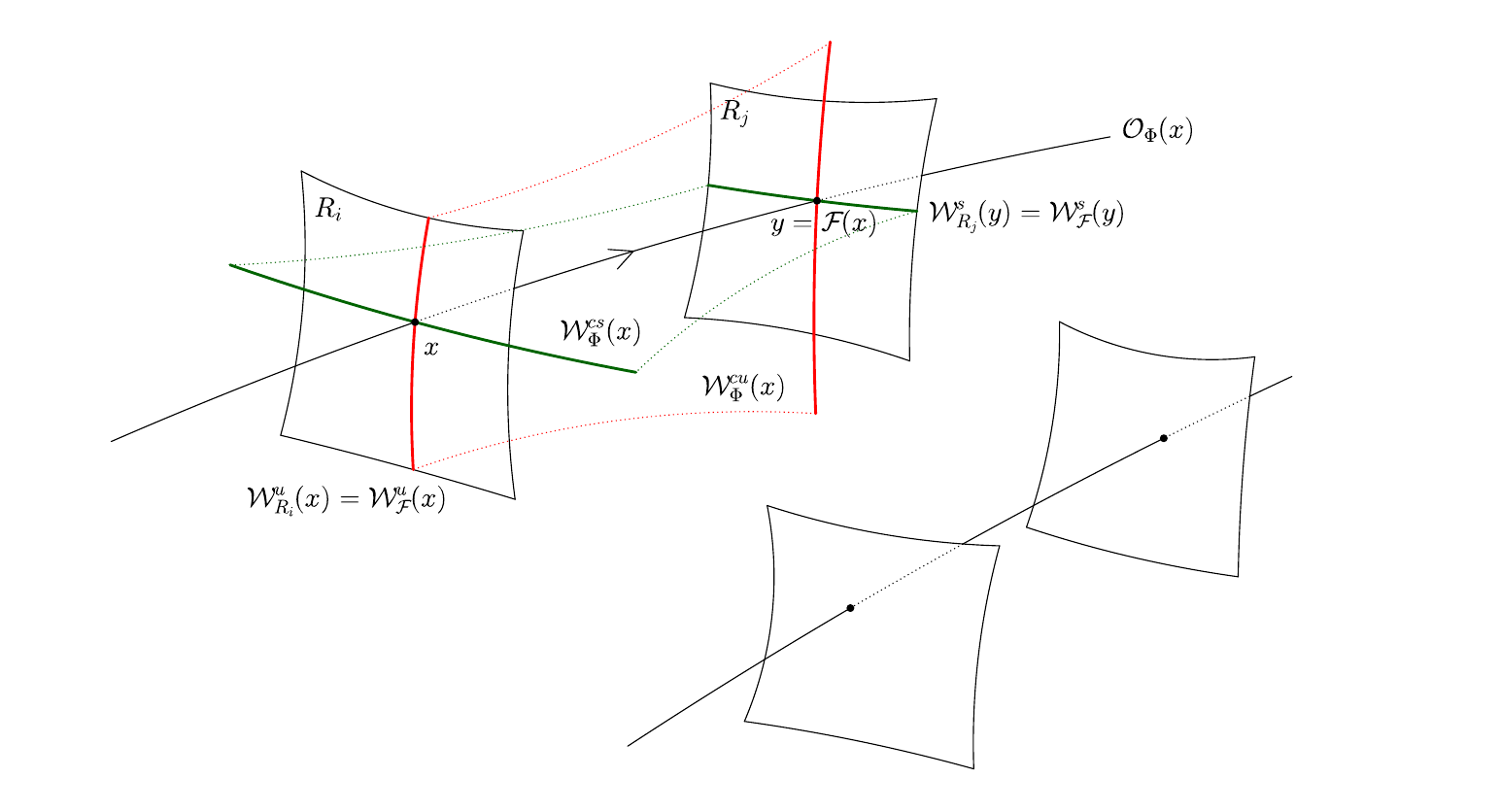}
	\caption{Markov family for the flow $\Phi$.}
\end{figure}

\subsection{Quadrilaterals and temporal displacements}

Let $\Phi=(\Phi^t)_{t \in \R}$ be a $\mathcal{C}^k$ Axiom A flow  on a smooth manifold $M$, with $k \geq 2$, and fix a basic set $\Lambda$ for $\Phi$.

\begin{defi}[Quadrilaterals]
	A \emph{quadrilateral} is a quadruple $\mathscr{Q}=(x_0,x_1,x_2,x_3)\subset \Lambda^4$ such that $x_1 \in \mathcal{W}_{\Phi,\mathrm{loc}}^s(x_0)$, $x_2 \in \mathcal{W}_{\Phi,\mathrm{loc}}^u(x_1)$ and $x_3 \in \mathcal{W}_{\Phi,\mathrm{loc}}^s(x_2)\cap \mathcal{W}_{\Phi,\mathrm{loc}}^{cu}(x_0)$. 
	We let $x_4=x_4(\mathscr{Q}):= \mathcal{W}_{\Phi,\mathrm{loc}}^{c}(x_0)\cap \mathcal{W}_{\Phi,\mathrm{loc}}^u(x_3)$. In particular, $x_4=\Phi^{t}(x_0)$, for some time $t=t(\mathscr{Q})\in \R$. 
\end{defi}

Let us consider a proper Markov family $\mathcal{R}=\{R_1,\dots,R_m\}$ for ${\Phi}_{|\Lambda}$ of size $\varepsilon$, for some integer $m\geq 1$ and some small $\varepsilon>0$. Let $\mathcal{F}$ be the associated Poincar\'e map, and set $\mathcal{S}:=R_1\cup \dots \cup R_m$.  We denote by  $\overline{\Lambda}:=\Lambda \cap \mathcal{S}$ the trace of $\Lambda$ on $\mathcal{S}$.  

We say that a quadrilateral $\mathscr{Q}=(x_0,x_1,x_2,x_3)\subset \Lambda^4$ is \emph{$\mathcal{R}$-good} if $x_0 \in R_i$ for some $i=i(\mathscr{Q})\in\{1,\dots,m\}$, and $x_j \in \cup_{t \in (-\frac \varepsilon 2,\frac \varepsilon 2)}\Phi^t(R_i)$, for each $j\in \{1,\dots,4\}$. Note that, up to time translation, there is no loss of generality to assume that $x_0 \in \mathcal{S}$. 
For any such quadrilateral, and for $j\in \{1,\dots,4\}$, we denote by $\bar x_j$ the projection along the flow line of $x_j$ on $R_i$, and we let $\overline{\mathscr{Q}}:=(\bar x_0,\bar x_1,\bar x_2,\bar x_3)$. Note that $\bar x_0,\dots,\bar x_3 \in \overline{\Lambda}$; besides, $\bar x_1 \in \mathcal{W}_{R_i}^s(\bar x_0)$, $\bar x_2 \in \mathcal{W}_{R_i}^u(\bar x_1)$, and $\bar x_3 \in \mathcal{W}_{R_i}^s(\bar x_2)\cap \mathcal{W}_{R_i}^{u}(\bar x_0)$.

\begin{defi}[s/u-holonomies]
	Fix $i\in\{1,\dots,m\}$, and let $z_0,z_1\in R_i\cap\Lambda$ be such that $z_1\in \mathcal{W}^{s}_{R_i}(z_0)$. We define the \emph{stable holonomy} $H_\mathcal{S}^s(z_0,z_1)\in \R$ as the time $t\in \R$ with smallest absolute value $|t|$ such that $\Phi^{t}(z_1)\in \mathcal{W}_{\Phi,\mathrm{loc}}^s(z_0)$.  
	Similarly, for any $z_0,z_1\in R_i\cap\Lambda$, $z_1\in \mathcal{W}^{u}_{R_i}(z_0)$, we define the \emph{unstable holonomy} $H_\mathcal{S}^u(z_0,z_1)\in \R$ as the time $t\in \R$ with smallest absolute value $|t|$ such that $\Phi^{t}(z_1)\in \mathcal{W}_{\Phi,\mathrm{loc}}^u(z_0)$.  
\end{defi}

\begin{lemma}\label{stable hol}
	For any $i\in\{1,\dots,m\}$, and for any $z_0,z_1 \in \mathcal{W}_{R_i}^s(z_0)$, it holds
	$$
	H_\mathcal{S}^s(z_0,z_1)=\sum_{j=0}^{+\infty} \tau_{\mathcal{S}}(\mathcal{F}^j(z_1))- \tau_{\mathcal{S}}(\mathcal{F}^j(z_0)). 
	$$
\end{lemma}

\begin{proof}
	Fix $i\in\{1,\dots,m\}$, and let $z_0,z_1\in R_i\cap\Lambda$ be such that $z_1\in \mathcal{W}^s_{R_i}(z_0)$. We abbreviate $H:=H_\mathcal{S}^s(z_0,z_1)$ and set $z_2=\Phi^H(z_1)$. Fix $\varepsilon>0$  arbitrarily small. As $z_1 \in \mathcal{W}_{R_i}^s(z_0)$ and  $z_2 \in \mathcal{W}_{\Phi,\mathrm{loc}}^s(z_0)$, for  $n\gg 1$ sufficiently large,   it holds 
	\begin{equation}\label{ineg hyper}
	\begin{array}{r}
	d(\mathcal{F}^n(z_0),\mathcal{F}^n(z_1))< \varepsilon,\\
	d(\Phi^{t_n}(z_0),\Phi^{t_n}(z_2))< \varepsilon,
	\end{array}
	\end{equation}
	with $\mathcal{F}^n(z_0)=\Phi^{t_n}(z_0)$ and $t_n:=\sum_{j=0}^{n-1} \tau_{\mathcal{S}}(\mathcal{F}^j(z_0))$. Set $u_n:=\sum_{j=0}^{n-1} \tau_{\mathcal{S}}(\mathcal{F}^j(z_1))$, so that $\mathcal{F}^n(z_1)=\Phi^{u_n}(z_1)$. The points $\mathcal{F}^n(z_0), \mathcal{F}^n(z_1)$ are exponentially close, and $\tau_\mathcal{S}$ is Lipschitz, hence the sequence $(u_n-t_n)_{n\geq 1}$ converges to some limit $\ell\in \R$. Since $z_2=\Phi^H(z_1)$, and by the triangular inequality, \eqref{ineg hyper} yields
	$$
	d(\Phi^{u_n}(z_1),\Phi^{t_n+H}(z_1))<2 \varepsilon. 
	$$ 
	As we are considering local manifolds, 
	we deduce that  
	$
	|u_n-t_n-H|<C \varepsilon
	$,
	for some uniform constant $C>0$. Letting $n \to +\infty$, we get $\ell=H$, i.e.,
	$$
	H=\sum_{j=0}^{+\infty} \tau_{\mathcal{S}}(\mathcal{F}^j(z_1))- \tau_{\mathcal{S}}(\mathcal{F}^j(z_0)). 
	$$
\end{proof}

Using the same ideas as in Lemma \ref{stable hol}, we have the following
\begin{lemma}\label{unstable hol}
	For any $i\in\{1,\dots,m\}$, and for any $z_0,z_1 \in \mathcal{W}_{R_i}^u(z_0)$, it holds
	$$
	H_\mathcal{S}^u(z_0,z_1)=\sum_{j=-\infty}^{-1} \tau_{\mathcal{S}}(\mathcal{F}^j(z_0))- \tau_{\mathcal{S}}(\mathcal{F}^j(z_1)). 
	$$
\end{lemma}

Let $\mathscr{Q}=(x_0,x_1,x_2,x_3)\subset \Lambda^4$ be a $\mathcal{R}$-good quadrilateral, with $x_0 \in R_i$, $i \in \{1,\dots,m\}$. Let $x_4=x_4(\mathscr{Q})$, and let $\overline{\mathscr{Q}}:=(\bar x_0,\bar x_1,\bar x_2,\bar x_3)$. As $\bar x_1 \in \mathcal{W}_{R_i}^s(\bar x_0)$, $\bar x_2 \in \mathcal{W}_{R_i}^u(\bar x_1)$, and $\bar x_3 \in \mathcal{W}_{R_i}^s(\bar x_2)\cap \mathcal{W}_{R_i}^{u}(\bar x_0)$, we may define the \emph{temporal displacement} $H(\mathscr{Q})\in \R$ as 
\begin{equation}\label{temporal displacemetn}
H(\mathscr{Q}):=H_\mathcal{S}^s(\bar x_0,\bar x_1)+H_\mathcal{S}^u(\bar x_1,\bar x_2)+H_\mathcal{S}^s(\bar x_2,\bar x_3)+H_\mathcal{S}^u(\bar x_3,\bar x_0).
\end{equation}
By Lemma \ref{stable hol} and \ref{unstable hol}, we have:
\begin{align}\label{limite H}
H(\mathscr{Q})&=\lim\limits_{n\to +\infty}\bigg[\sum_{j=-n}^{n} -\tau_{\mathcal{S}}(\mathcal{F}^j(\bar x_0))+ \tau_{\mathcal{S}}(\mathcal{F}^j(\bar x_1))-\tau_{\mathcal{S}}(\mathcal{F}^j(\bar x_2))+\tau_{\mathcal{S}}(\mathcal{F}^j(\bar x_3))\bigg]\nonumber \\
&=\lim_{n\to +\infty} \big[-\tau_\mathcal{S}^n(\bar x_0)+\tau_\mathcal{S}^n(\bar x_1)-\tau_\mathcal{S}^n(\bar x_2)+\tau_\mathcal{S}^n(\bar x_3)\big],
\end{align}
where for any point $z \in \mathcal{S}$, and for any integer $n \geq 0$, we let 
\begin{equation}\label{itere tau ciao}
\tau_\mathcal{S}^n(z):=\sum_{j=-n}^{n} \tau_{\mathcal{S}}(\mathcal{F}^j(z)).
\end{equation}

\begin{figure}[h]
	\includegraphics[scale=1, trim=0 1cm 0 0, clip]{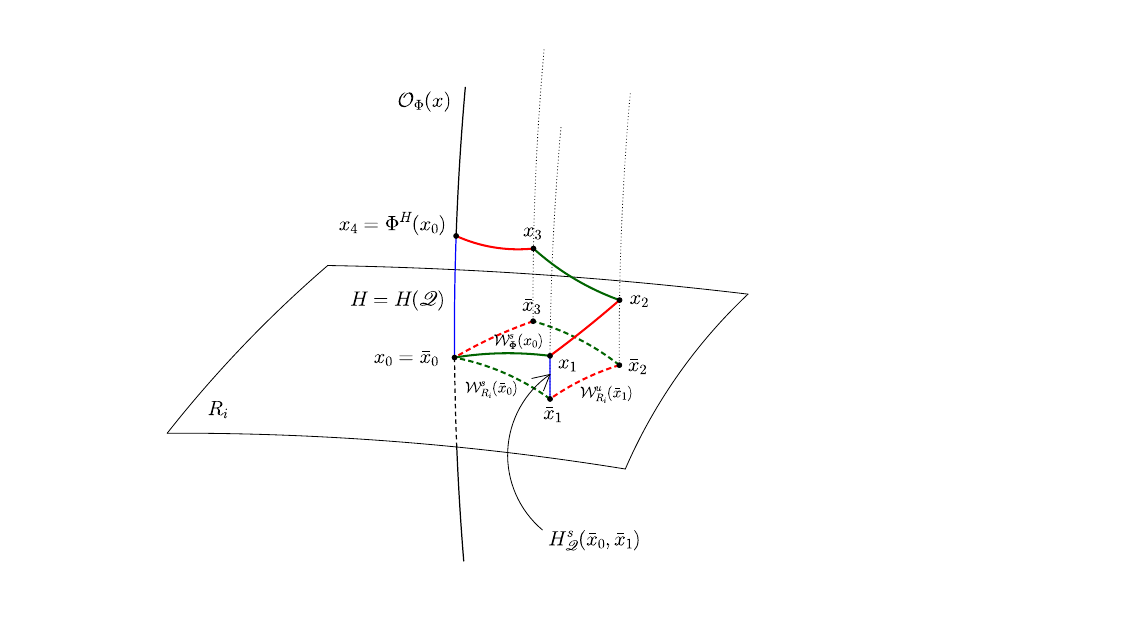}
	\caption{Quadrilaterals and temporal displacements.}
\end{figure}

\subsection{Periodic approximations of temporal displacements}

Let us recall the following fact. 
\begin{lemma}\label{continuity holonomies}
	For each $i \in \{1,\dots,m\}$, the stable holonomies $H_\mathcal{S}^s(y_0,y_1)$, resp. unstable holonomies $H_\mathcal{S}^u(z_0,z_1)$, depend continuously on the points $y_0, y_1 \in R_i\cap \Lambda$, $y_1 \in \mathcal{W}_{R_i}^{s}(y_0)$, resp. on the points $z_0, z_1 \in R_i\cap \Lambda$, $z_1 \in \mathcal{W}_{R_i}^{u}(z_0)$. 
\end{lemma}

\begin{proof}
	Let us consider the case where $y_0, y_1 \in R_i\cap \Lambda$, $y_1 \in \mathcal{W}_{R_i}^{s}(y_0)$, the other case is analogous. By definition, the stable holonomy $H_\mathcal{S}^s(y_0,y_1)$ satisfies 
	$$
	\mathcal{W}_{\Phi,\mathrm{loc}}^s(y_0)\cap \mathcal{W}_{\Phi,\mathrm{loc}}^c(y_1)=\{\Phi^{H_\mathcal{S}^s(y_0,y_1)}(y_1)\}.
	$$
	As the invariant manifolds vary continuously, the intersection of the two sets on the left hand side depends continuously on the pair $y_0,y_1$, with $y_1 \in \mathcal{W}_{R_i}^{s}(y_0)$. By looking at the right hand side, we conclude that the holonomies are continuous. 
\end{proof}

The main goal of this section is to show the following proposition, whose content already appears in the work of Otal \cite{Otal}. 
\begin{prop}\label{main propo}
	For any $\mathcal{R}$-good quadrilateral $\mathscr{Q}=(x_0,x_1,x_2,x_3)\in \Lambda^4$, the quantity  $H(\mathscr{Q})$ is determined by the lengths of periodic orbits.
\end{prop}

Proposition \ref{main propo} is a direct outcome of Lemma \ref{lemma approx} and Proposition \ref{main propo restr} below. 

\begin{lemma}\label{lemma approx}
For any $\mathcal{R}$-good quadrilateral $\mathscr{Q}=(x_0,x_1,x_2,x_3)\in \Lambda^4$, there exists a sequence $(\mathscr{Q}^n)_{n \in \N}\in(\Lambda^4)^\N$  of $\mathcal{R}$-good quadrilaterals $\mathscr{Q}^n=(x_0^n,x_1^n,x_2^n,x_3^n)$ with $x_0^n,x_2^n \in \mathrm{Per}(\Phi)$ such that $\lim_{n \to +\infty} \mathscr{Q}^n=\mathscr{Q}$, i.e., $\lim_{n \to\infty} x_j^n=x_j$, for each $j=0,\dots,3$. In particular, it holds 
$$
H(\mathscr{Q})=\lim_{n \to +\infty}H(\mathscr{Q}^n). 
$$
\end{lemma}

\begin{proof}
	Fix a $\mathcal{R}$-good quadrilateral $\mathscr{Q}=(x_0,x_1,x_2,x_3)\in \Lambda^4$, with $x_0 \in R_i$, $i \in \{1,\dots,m\}$, and let $\overline{\mathscr{Q}}:=(\bar x_0,\bar x_1,\bar x_2,\bar x_3)$ be the projection of $\mathscr{Q}$ on $R_i$ as before. 
	
	As periodic points are dense in $\Lambda$, for $j=0,2$, there exists a sequence $(\bar x_j^n)_{n \in \N}\in \big(\mathrm{Per}(\Phi)\cap R_i\big)^\N$ of periodic points such that $\lim_{n\to +\infty} \bar x_j^n=\bar x_j$. Let $\bar x_1^n:=[\bar x_0^n,\bar x_2^n]_{R_i}$ and $\bar x_3^n:=[\bar x_2^n,\bar x_0^n]_{R_i}$, so that the lift $\mathscr{Q}^n:=(x_0^n,x_1^n,x_2^n,x_3^n)$ of $\overline{\mathscr{Q}}^n:=(\bar x_0^n,\bar x_1^n,\bar x_2^n,\bar x_3^n)$  is a $\mathcal{R}$-good quadrilateral, where
	\begin{align*}
	x_0^n&:=\bar x_0^n,\quad & x_1^n&:=\Phi^{H_\mathcal{S}^s(\bar x_0^n,\bar x_1^n)}(\bar x_1^n),\\ x_2^n&:=\Phi^{H_\mathcal{S}^s(\bar x_0^n,\bar x_1^n)+H_\mathcal{S}^u(\bar x_1^n,\bar x_2^n)}(\bar x_2^n), & x_3^n&:=\Phi^{H_\mathcal{S}^s(\bar x_0^n,\bar x_1^n)+H_\mathcal{S}^u(\bar x_1^n,\bar x_2^n)+H_\mathcal{S}^s(\bar x_2^n,\bar x_3^n)}(\bar x_3^n),
	\end{align*}
	and $x_0^n,x_2^n\in \mathrm{Per}(\Phi)$. 
	Clearly, we have $\lim_{n\to +\infty}\mathscr{Q}^n=\mathscr{Q}$. 
	By the definition \eqref{temporal displacemetn} of temporal displacements in terms of holonomies, and by Lemma \ref{continuity holonomies}, the function $\widetilde{\mathscr{Q}}\mapsto H(\widetilde{\mathscr{Q}})$ is continuous. Thus, we conclude that $H(\mathscr{Q})=\lim_{n \to +\infty}H(\mathscr{Q}^n)$. 
\end{proof}
	
\begin{prop}\label{main propo restr}
For any $\mathcal{R}$-good quadrilateral $\mathscr{Q}=(x_0,x_1,x_2,x_3)\in \Lambda^4$ such that $x_0,x_2\in \mathrm{Per}(\Phi)$, the quantity  $H(\mathscr{Q})$ is determined by the lengths of periodic orbits. More precisely,  there exists a sequence $(\bar x^n)_{n \in \N}\in \mathrm{Per}(\Phi)^\N$ of 
periodic points 
such that for any $\varepsilon>0$, there exists an integer $N_0(\varepsilon)\in \N$  such that 
$$
\Big|H(\mathscr{Q})-  \big[T_\Phi(\bar x^n)-(4n+1) T_\Phi(x_0)-(4n+1) T_\Phi(x_2)\big]\Big|< \varepsilon,\quad \forall\, n \geq N_0(\varepsilon). 
$$
\end{prop}

\begin{proof}
	Fix a $\mathcal{R}$-good quadrilateral $\mathscr{Q}=(x_0,x_1,x_2,x_3)\in \Lambda^4$, with $x_0 \in R_i$, $i \in \{1,\dots,m\}$ and $x_0,x_2\in \mathrm{Per}(\Phi)$, and let $\overline{\mathscr{Q}}:=(\bar x_0,\bar x_1,\bar x_2,\bar x_3)$ be the projection of $\mathscr{Q}$ on $R_i$. 
	Let us choose $p\geq 1$ sufficiently large such that $\mathcal{F}^p(\bar x_0)=\bar x_0$ and $\mathcal{F}^p(\bar x_2)=\bar x_2$. After replacing $\tau_\mathcal{S}$ with 
	\begin{equation*}
	\tau_{+,\mathcal{S}}^p(\cdot):=\sum_{j=0}^{p-1} \tau_{\mathcal{S}}\circ\mathcal{F}^j(\cdot),
	\end{equation*} 
	we may thus assume that $x_0,x_2$ are fixed points under the Poincar\'e map $\mathcal{F}$. 
	
	Note that $\bar x_1=[\bar x_0,\bar x_2]_{R_i}$ and $\bar x_3=[\bar x_2,\bar x_0]_{R_i}$ are heteroclinic intersections between the invariant manifolds of the fixed points $\bar x_0,\bar x_2$. 
	
	\begin{figure}[h]
		\includegraphics[scale=0.8, trim=0 1.5cm 0 1.6cm, clip]{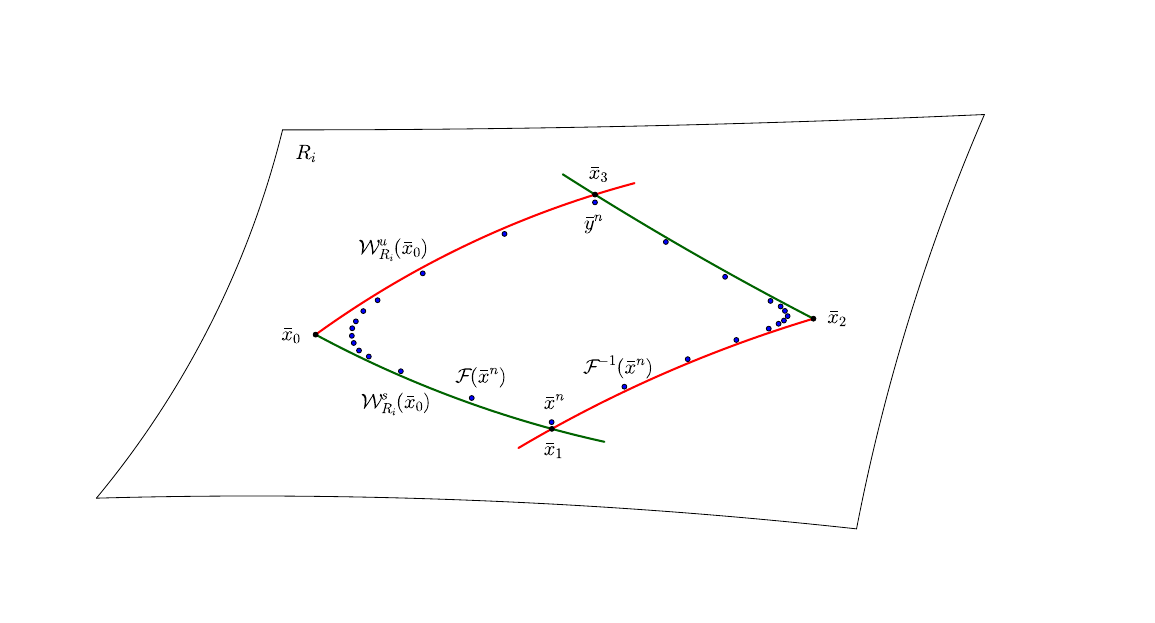}
		\caption{Approximating periodic orbits with a prescribed combinatorics.}
	\end{figure}

As $\Lambda$ is a basic set, the dynamics can be coded symbolically, using some finite alphabet $\mathscr{A}$. In the following, for each finite word $\sigma$ in $\mathscr{A}$, we denote by $|\sigma|\in \N$ the length of $\sigma$. 

The fixed (periodic) points $\bar x_0$, $\bar x_2$, correspond to a symbol (a finite sequence of symbols) $\sigma_0$, $\sigma_2$ respectively. The point $\bar x_1$ is a heteroclinic intersection between $\mathcal{W}_{R_i}^s(\bar x_0)$ and $\mathcal{W}_{R_i}^u(\bar x_2)$, hence there exist a symbol $\sigma_0^1 \in \mathscr{A}$ and two finite words $\sigma_-^1, \sigma_+^1$ in $\mathscr{A}$ such that the symbolic coding of $\bar x_1$ is 
$$
\bar x_1 \longleftrightarrow \dots\sigma_2\sigma_2\sigma_-^1\underset{\substack{\uparrow}}{\sigma_0^1}\sigma_+^1\sigma_0\sigma_0\dots
$$
Similarly, there exist a symbol $\sigma_0^3 \in \mathscr{A}$ and two finite words $\sigma_-^3, \sigma_+^3$ in $\mathscr{A}$ such that the symbolic coding of $\bar x_3$ is 
$$
\bar x_3 \longleftrightarrow\dots\sigma_0\sigma_0\sigma_-^3\underset{\substack{\uparrow}}{\sigma_0^3}\sigma_+^3\sigma_2\sigma_2\dots
$$
Up to redefining $\sigma_+^1$ as $\sigma_+^1\underbrace{\sigma_0\dots\sigma_0}_{n-\vert\sigma_+^1\vert}$, without loss of generality, we can assume that $\vert\sigma_+^1\vert=n$. 
Similarly, we can assume that $\vert\sigma_-^1\vert=\vert\sigma_+^3\vert=\vert\sigma_-^3\vert=n$.

For each integer $n \geq 0$, we define a periodic point $\bar x^n$ whose symbolic coding is given by the infinite word $\dots(\sigma_0^1\sigma^n)(
\underset{\substack{\uparrow}}{\sigma_0^1}\sigma^n)
(\sigma_0^1\sigma^n)\dots$, where  
$$
\sigma^n:=\sigma_+^1\underbrace{\sigma_0\dots\sigma_0}_{2n}\sigma_-^3\sigma_0^3\sigma_+^3\underbrace{\sigma_2\dots\sigma_2}_{2n}\sigma_-^1.
$$
Thus, the point of $\bar x^n$ is a periodic point, of period
$$
2+4n+|\sigma_-^1|+|\sigma_+^1|+|\sigma_-^3|+|\sigma_+^3|=2+8n. 
$$
\begin{lemma}\label{claim chiante}
	For any $\varepsilon>0$, there exists $N=N(\varepsilon)>0$ such that for each integer $n \geq N$, the following inequalities hold:
	\begin{align}
	\left| \sum_{k=-2n}^{2n}\Big[\tau_\mathcal{S}(\mathcal{F}^k(\bar x^n))-\tau_\mathcal{S}(\mathcal{F}^k(\bar x_1))\Big]\right|&<\varepsilon,\label{eq 1 ineq}\\
	\left| \sum_{k=-2n}^{2n}\Big[\tau_\mathcal{S}(\mathcal{F}^{k}(\bar y^n))-\tau_\mathcal{S}(\mathcal{F}^k(\bar x_3))\Big]\right|&<\varepsilon,\label{eq 2 ineq}
	\end{align}
	where we have set $\bar y^n:=\mathcal{F}^{4n+1}(\bar x^n)$, and for each $M \geq 2n+1$, we have 
	\begin{align}
	\left| \sum_{k=2n+1}^{M}\Big[\tau_\mathcal{S}(\mathcal{F}^k(\bar x_1))-T_\Phi(\bar x_0)\Big]\right|+\left| \sum_{k=-M}^{-2n-1}\Big[\tau_\mathcal{S}(\mathcal{F}^k(\bar x_1))-T_\Phi(\bar x_2)\Big]\right|&<\varepsilon,\label{eq 3 ineq}\\
	\left| \sum_{k=2n+1}^{M}\Big[\tau_\mathcal{S}(\mathcal{F}^k(\bar x_3))-T_\Phi(\bar x_2)\Big]\right|+\left| \sum_{k=-M}^{-2n-1}\Big[\tau_\mathcal{S}(\mathcal{F}^k(\bar x_3))-T_\Phi(\bar x_0)\Big]\right|&<\varepsilon. \label{eq 4 ineq}
	\end{align}
	Observe that the two sums in \eqref{eq 1 ineq} and in \eqref{eq 2 ineq} add up to
	\begin{equation}\label{equation point per}
	\sum_{k=-2n}^{2n}\tau_\mathcal{S}(\mathcal{F}^k(\bar x^n))+
	\sum_{k=-2n}^{2n}\tau_\mathcal{S}(\mathcal{F}^{k}(\bar y^n))=T_\Phi(\bar x^n). 
	\end{equation}
\end{lemma}

\begin{proof}[Proof of Lemma \ref{claim chiante}]
	By looking at the symbolic codings of $\bar x^n$ and $\bar x_1$, we see that they have the same symbolic past (resp. future) for at least $3n$ steps of iterations under  $\mathcal{F}$. By hyperbolicity of $\mathcal{F}$, for some constant $\lambda\in (0,1)$, we thus have 
	\begin{equation}\label{proximite points}
	d\big(\mathcal{F}^k(\bar x^n),\mathcal{F}^k(\bar x_1)\big)=O(\lambda^n),\quad \forall\, k\in \{-2n,\dots,2n\}. 
	\end{equation}
	Indeed, without loss of generality (after possibly iterating $n_0$ times, for some integer $n_0\geq 1$ independent of $n$), we may assume that each of these points belongs to some small neighborhood of $\bar x_0$  where the dynamics is conjugated to the differential $D\mathcal{F}(\bar x_0)$. More precisely,  by Lemma 23  in \cite{HKS}, for any $\delta>0$, and for $j=0,2$, 
	there exist a neighborhood $\mathcal{U}_j$ of $\bar x_j$, a neighborhood $\mathcal{V}_j \subset \R^2$ of $(0,0)$, and a $\mathcal{C}^{1,\frac 12}$-diffeomorphism  $\chi_j\colon\mathcal{U}_j \to  \mathcal{V}_j$, such that
	\begin{equation*}
	\chi_j\circ \mathcal{F} \circ \chi_j^{-1} = D\mathcal{F}(\bar x_j),\qquad \|\chi_j-\mathrm{id}\|_{\mathcal{C}^1} \leq \delta, \qquad \|\chi_j^{-1}-\mathrm{id}\|_{\mathcal{C}^1} \leq \delta. 
	\end{equation*}
	
	By \eqref{proximite points}, summing over all the indices $k\in \{-2n,\dots,2n\}$, and as $\tau_\mathcal{S}$ is Lipschitz continuous, the left hand side in \eqref{eq 1 ineq} is of order at most $O(n\lambda^n)$; therefore, for $n$ sufficiently large, this term is smaller than $\varepsilon$. Inequality \eqref{eq 2 ineq} is proved similarly.
	
	Finally, \eqref{eq 3 ineq} is proved using the same linearizing coordinates near $\bar x_0$ and $\bar x_2$, noting that $\bar x_1 \in \mathcal{W}_{R_i}^s(\bar x_0)$, resp. $\bar x_1 \in \mathcal{W}_{R_i}^u(\bar x_2)$, so that $\bar x_1$ has the same future as $\bar x_0$, resp. the same past as $\bar x_2$, hence all of its future, resp. past iterates (after iterating finitely many times) belong to the neighborhood $\mathcal{U}_0$ of $\bar x_0$, resp. to the neighborhood $\mathcal{U}_2$ of $\bar x_2$, endowed with linearizing coordinates. We argue similarly for \eqref{eq 4 ineq}, which concludes the proof of Lemma \ref{claim chiante}. 
\end{proof}

Let us now conclude the proof of Proposition \ref{main propo restr}. Fix some small $\varepsilon>0$. By \eqref{limite H}, for $m \geq 1$ sufficiently large, we have
\begin{equation*}
\Big|H(\mathscr{Q})-\sum_{j=-m}^{m} \Big[-\tau_{\mathcal{S}}(\mathcal{F}^j(\bar x_0))+ \tau_{\mathcal{S}}(\mathcal{F}^j(\bar x_1))-\tau_{\mathcal{S}}(\mathcal{F}^j(\bar x_2))+\tau_{\mathcal{S}}(\mathcal{F}^j(\bar x_3))\Big]\Big| < \frac \varepsilon 2. 
\end{equation*}

By Lemma \ref{claim chiante}, there exists a periodic point $\bar x^n\in \mathrm{Per}(\Phi)$ such that inequalities \eqref{eq 1 ineq},\eqref{eq 2 ineq},\eqref{eq 3 ineq},\eqref{eq 4 ineq} hold for $\bar x^n$ and $\frac \varepsilon 8$ in place of $\varepsilon$. Splitting the different sums of return times to match these inequalities, and thanks to \eqref{equation point per}, we conclude that 
$$
\Big| H(\mathscr{Q})-\big[T_\Phi(\bar x^n)-(4n+1) T_\Phi(x_0)-(4n+1) T_\Phi(x_2)\big]\Big|< \varepsilon,
$$
as desired. 
\end{proof}

\subsection{Temporal displacements and areas of quadrilaterals}

Assume that there exists a smooth contact form $\alpha$ on $M$ that is adapted to the basic set $\Lambda$ in the sense of Definition \ref{adapted contact form}. 
Recall the following fact:
\begin{lemma}\label{lemma noyau ker}
	We have $E_\Phi^s(x)\subset \ker \alpha(x)$, for all $x \in \mathcal{W}_\Phi^s(\Lambda)$,  and $E_\Phi^u(x)\subset \ker \alpha(x)$, for all $x \in \mathcal{W}_\Phi^u(\Lambda)$. In particular, it holds 
	\begin{equation}\label{noyau alpha}
	E_\Phi^s(x)\oplus E_\Phi^u(x)= \ker \alpha(x),\quad \forall\, x \in \Lambda. 
	\end{equation}
\end{lemma}

\begin{proof}
	Let $\Gamma=\{\gamma(t)\in t \in [0,1]\}\subset \mathcal{W}_{\Phi,\mathrm{loc}}^s(x)$ be an arc in the local stable manifold of some point $x \in \Lambda$.  For each $T >0$, we have
	$$
	\int_\Gamma \alpha=\int_{0}^1 \alpha(\gamma(t))( \gamma'(t))dt=\int_0^1   \alpha(\Phi^T \circ \gamma(t))(D\Phi^T(\gamma(t))\cdot \gamma'(t))dt =\int_{\Phi^T\circ \Gamma}  \alpha. 
	$$
	As $\alpha$ is uniformly bounded, and $\lim_{T \to +\infty}D\Phi^T(\gamma(t))\cdot \gamma'(t) \to 0$, for each $t \in [0,1]$, we deduce that 
	$\int_\Gamma \alpha=0$. 
	Therefore, we have $E_\Phi^s(y) \subset \ker \alpha(y)$, for any $y \in \mathcal{W}_\Phi^s(x)$, $x \in \Lambda$. We argue similarly for the unstable direction. 
	
	Let $x \in \Lambda$. The identity \eqref{noyau alpha} follows from the inclusions $E_\Phi^s(x)\subset \ker \alpha(x)$, $E_\Phi^u(x)\subset \ker \alpha(x)$, and the equality of the dimensions of the two subspaces. 
\end{proof}

Let $\mathscr{Q}=(x_0,x_1,x_2,x_3)\in \Lambda^4$ be a $\mathcal{R}$-good quadrilateral, with $x_0 \in R_i$, for some $i \in \{1,\dots,m\}$, and let $\overline{\mathscr{Q}}:=(\bar x_0,\bar x_1,\bar x_2,\bar x_3)$ be the projection of $\mathscr{Q}$ on $R_i$. We define $\widehat{\mathscr{Q}}$ as the set of all points $x \in R_i$ in the closed region bounded by the arcs $\bar \Gamma_0$, $\bar \Gamma_1$, $\bar \Gamma_2$, $\bar \Gamma_3$, where for $j=0,2$, $\bar \Gamma_j\subset \mathcal{W}_{R_i}^s(\bar x_j)$ is the stable arc connecting $\bar x_j$ to $\bar x_{j+1}$, while  $\bar \Gamma_{j+1}\subset \mathcal{W}_{R_i}^u(\bar x_{j+1})$ is the unstable arc connecting $\bar x_{j+1}$ to $\bar x_{j+2}$, with $\bar x_4:=\bar x_0$. The set $\widehat{\mathscr{Q}} \subset R_i$ is transverse to the flow direction, i.e., 
\begin{equation}\label{cond un peu chiante}
	X(x) \notin T_x \widehat{\mathscr{Q}},\text{ for each } x \in \widehat{\mathscr{Q}},
\end{equation}
which ensures that $d\alpha|_{\widehat{\mathscr{Q}}}$ is non-degenerate.
Let us define 
$$
\mathrm{Area}(\mathscr{Q}):=\int_{\widehat{\mathscr{Q}}} d\alpha. 
$$

\begin{prop}\label{lien h aire}
	Let $\mathscr{Q}=(x_0,x_1,x_2,x_3)\in \Lambda^4$ be a small quadrilateral, so that $\mathscr{Q}$ is  $\mathcal{R}$-good and \eqref{cond un peu chiante} is satisfied. Then
	$$
	\mathrm{Area}(\mathscr{Q})=-H(\mathscr{Q}). 
	$$
\end{prop}

\begin{proof}
	By Stokes theorem, we have 
	$$
	\mathrm{Area}(\mathscr{Q})=\int_{\widehat{\mathscr{Q}}} d\alpha=\sum_{j=0,\dots,3}\int_{\bar \Gamma_j} \alpha.
	$$
	By the definition \eqref{temporal displacemetn} of $H(\mathscr{Q})$ in terms of holonomies, it is sufficient to show that $\int_{\bar \Gamma_0} \alpha=-H_{\mathcal{S}}^s(\bar x_0,\bar x_1)$, $\int_{\bar \Gamma_1} \alpha=-H_{\mathcal{S}}^u(\bar x_1,\bar x_2)$, $\int_{\bar \Gamma_2} \alpha=-H_{\mathcal{S}}^s(\bar x_2,\bar x_3)$, and  $\int_{\bar \Gamma_3} \alpha=-H_{\mathcal{S}}^u(\bar x_3,\bar x_0)$. Let us prove the formula for $\bar\Gamma_0$, the others are proved similarly. 
	
	\begin{figure}[h]
		\includegraphics[scale=1, trim=0 1.2cm 0 0.5cm, clip]{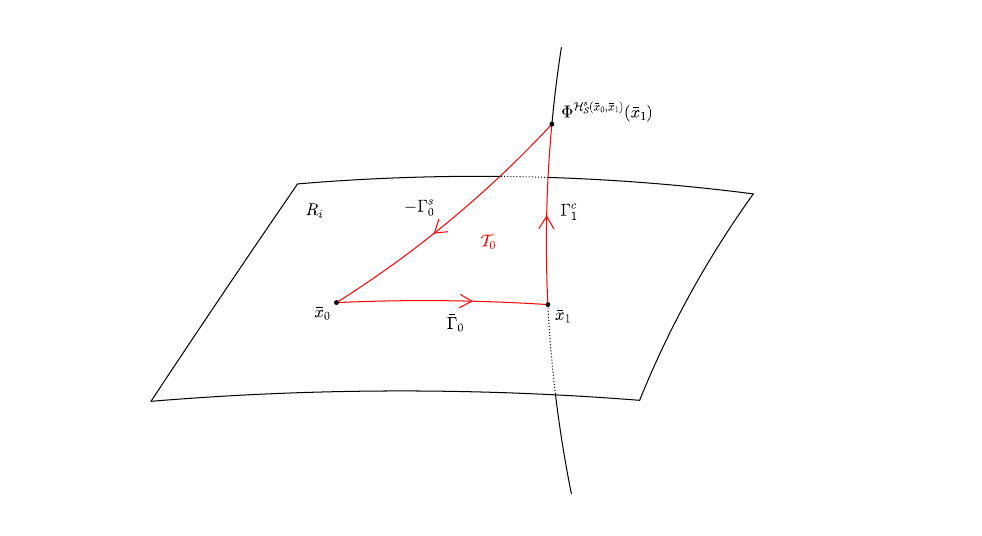}
		\caption{$\mathcal{T}_0$ is the closed region bounded by the arcs $\bar{\Gamma}_0,-\Gamma_0^s,\Gamma_1^c$.}
		\label{Fig T0}
	\end{figure}
	
	 Let $\Gamma_0^s$ be the arc of the stable manifold $\mathcal{W}_{\Phi,\mathrm{loc}}^s(x_0)$ connecting $x_0$ to $x_1$, and let $\Gamma_1^c$ be  the orbit segment $\Gamma_1^c:=\{\Phi^t(x_1)\}_{t \in [0,\mathcal{H}_\mathcal{S}^s(\bar x_0,\bar x_1)]}\subset \mathcal{W}_{\Phi,\mathrm{loc}}^c(x_1)$. We define $\mathcal{T}_0\subset M$ as the set of all points $x \in \mathcal{W}_{\Phi, \mathrm{loc}}^{cs}(x_0)$ in the closed region bounded by the arcs $\bar \Gamma_0,\Gamma_0^s,\Gamma_1^c$, see Figure \ref{Fig T0}. By Stokes theorem, we have 
	 $$
	 \int_{\mathcal{T}_0} d\alpha=\int_{\bar \Gamma_0} \alpha-\int_{\Gamma_0^s} \alpha+ \int_{\Gamma_1^c}\alpha. 
	 $$
	 Since $X|_{\mathcal{W}_\Phi^{cs}(\Lambda)} \in \ker d\alpha|_{\mathcal{W}_\Phi^{cs}(\Lambda)}$, it holds that $\int_{\mathcal{T}_0} d\alpha=0$. By Lemma \ref{lemma noyau ker},  we have $\int_{\Gamma_0^s} \alpha=0$, hence, 
	 $$
	 \int_{\bar \Gamma_0} \alpha=-\int_{\Gamma_1^c}\alpha.
	 $$ 
	 Moreover, 
	 $$
	 \int_{\Gamma_1^c}\alpha=\int_{0}^{H_{\mathcal{S}}^s(\bar x_0,\bar x_1)} \alpha(X(\Phi^t(\bar x_1))) dt. 
	 $$
	 Since $\imath_{X}\alpha|_{\Lambda}\equiv 1$, we also have $\int_{0}^{H_{\mathcal{S}}^s(\bar x_0,\bar x_1)} \alpha(X(\Phi^t(\bar x_1))) dt=H_{\mathcal{S}}^s(\bar x_0,\bar x_1)$, which concludes. 
\end{proof}

As an immediate consequence of Proposition \ref{main propo} and Proposition \ref{lien h aire}, we thus obtain:
\begin{corollary}
	For any small quadrilateral $\mathscr{Q}=(x_0,x_1,x_2,x_3)\in \Lambda^4$, the quantity  $\mathrm{Area}(\mathscr{Q})$ is determined by the lengths of periodic orbits. 
\end{corollary}

\begin{corollary}\label{egalite aires}
	Fix $k \geq 2$. 
	For $i=1,2$, let $\Phi_i=(\Phi_i^t)_{t \in \R}$ be a $\mathcal{C}^k$ Axiom A flow defined on a $3$-manifold $M_i$. Let $\Lambda_i$ be a basic set for $\Phi_i$, and let $\alpha_i$ be a smooth contact form adapted to $\Lambda_i$. If there exists an iso-length-spectral flow conjugacy $\Psi\colon \Lambda_1 \to  \Lambda_2$ between $\Phi_1|_{\Lambda_1}$ and $\Phi_2|_{\Lambda_2}$, then, for any point $x_0 \in \Lambda_1$, and for any small quadrilateral $\mathscr{Q}=(x_0,x_1,x_2,x_3)\in \Lambda_1^4$, it holds 
	$$
	\mathrm{Area}(\mathscr{Q})=\mathrm{Area}(\Psi(\mathscr{Q})),
	$$
	where $\Psi(\mathscr{Q})$ is the quadrilateral  $\Psi(\mathscr{Q}):=(\Psi(x_0),\Psi(x_1),\Psi(x_2),\Psi(x_3))\in \Lambda_2^4$. 
\end{corollary}

\subsection{Smoothness of the conjugacy}

In the following, we fix a point $x_0 \in \Lambda\cap R_i$, for some $i \in \{1,\dots,m\}$.  Let $Q_0$ be the set of all sufficiently small quadrilaterals $\mathscr{Q}=(x_0,x_1,x_2,x_3)\in \Lambda^4$ based at $x_0$. 
The goal of this part is to show that 
the set of areas $\{\mathrm{Area}(\mathscr{Q})\}_{\mathscr{Q}\in Q_0}$ determines the ``infinitesimal'' shape of the set $\Lambda\cap \mathcal{W}_{{\mathcal{F}},\mathrm{loc}}^s(x_0)$, resp.  $\Lambda\cap \mathcal{W}_{\mathcal{F},\mathrm{loc}}^u(x_0)$.

In particular, given another Axiom A flow whose restriction to some basic set is conjugate to $\Phi|_\Lambda$ by some homeomorphism $\Psi$, and such that, for any small quadrilateral $\mathscr{Q}$, it holds $\mathrm{Area}(\mathscr{Q})=\mathrm{Area}(\Psi(\mathscr{Q}))$, we show that $\Psi$ is 
differentiable at any point of $\Lambda$, with H\"older continuous differential.

\begin{figure}[h]
	\includegraphics[scale=0.85, trim=0 2cm 0 2cm, clip]
	{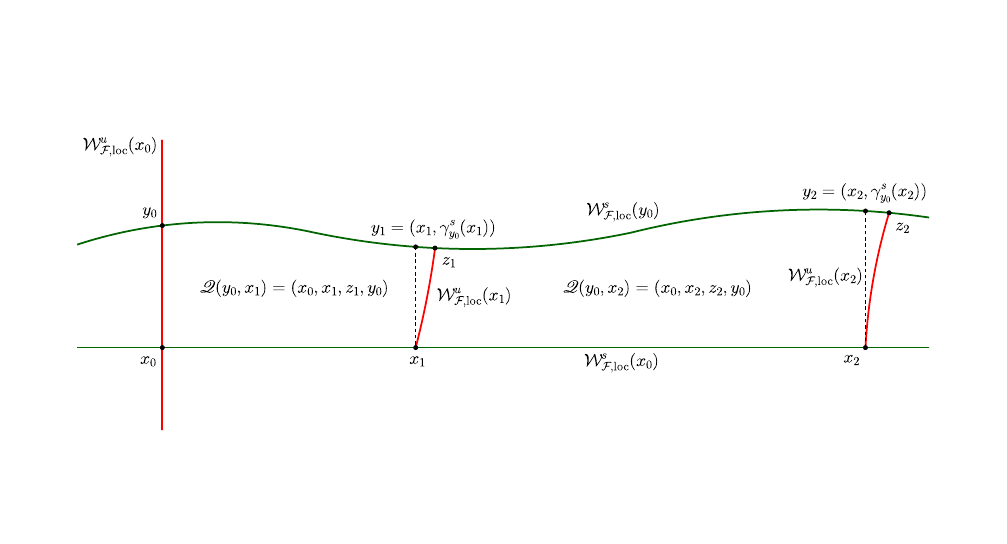}
	\caption{Small quadrilaterals.}
\end{figure}

We take a chart $\mathcal{R}=\mathcal{R}_{x_0} \colon \mathcal{U}_0 \to \mathcal{V}_0$ from a neighborhood $\mathcal{U}_0 \subset R_i$ of $x_0$ to a neighborhood $\mathcal{V}_0 \subset \R^2$ of $\{0_{\R^2}\}$ such that $\mathcal{R}(\mathcal{W}_{\mathcal{F},\mathrm{loc}}^s(x_0)\cap \mathcal{U}_0 )\subset (\R \times \{0\})\cap \mathcal{V}_0$ and $\mathcal{R}(\mathcal{W}_{\mathcal{F},\mathrm{loc}}^u(x_0)\cap \mathcal{U}_0 )\subset (\{0\}\times \R)\cap \mathcal{V}_0$. In the following, we thus identify  $\mathcal{W}_{\mathcal{F},\mathrm{loc}}^s(x_0)$, resp. $\mathcal{W}_{\mathcal{F},\mathrm{loc}}^u(x_0)$ with the horizontal, resp. vertical coordinate axis of  $\R^2$. 
Moreover, for any point $v= \mathcal{R}(u) \in  \mathcal{V}_0$, we denote by $\rho(v) d\xi\wedge d\eta:=\mathcal{R}_* (d\alpha_u)$ the corresponding area form. For each point $y_0 \in \mathcal{W}_{\mathcal{F},\mathrm{loc}}^u(x_0)$, we see $\mathcal{W}_{\mathcal{F},\mathrm{loc}}^s(y_0)$ as the graph of some function $\gamma_{y_0}^s$ over the horizontal axis. By an abuse of notation, in the following, we identify an object and its image in the chart $\mathcal{R}$. For instance, a point $x_1 \in \mathcal{W}_{\mathcal{F},\mathrm{loc}}^s(x_0)$ will be identified with the point $(x_1,0)\in \R^2$; besides, we will denote by $V_{x_1}:=\big(\{x_1\}\times \R\big) \cap \mathcal{V}_0$ the vertical segment in $\mathcal{V}_0$ passing through $x_1\simeq (x_1,0)$.

\begin{defi}[Holonomy maps for $\mathcal{W}_{\mathcal{F},\mathrm{loc}}^s$] 
	For 
	any points  $y_0 \in \mathcal{W}_{\mathcal{F},\mathrm{loc}}^u(x_0)\cap \Lambda$ and $x_1 \in \mathcal{W}_{\mathcal{F},\mathrm{loc}}^s(x_0)$, we define the point $\mathcal{H}^s_{x_0,x_1}(y_0)\in V_{x_1}$ as 
	$$
	\{\mathcal{H}^s_{x_0,x_1}(y_0)\}:=  
	\mathcal{W}_{\mathcal{F},\mathrm{loc}}^s(y_0) \cap V_{x_1}=(x_1,\gamma_{y_0}^s(x_1)). 
	$$
\end{defi}

In other words, the map $\mathcal{H}^s_{x_0,x_1}$ is the holonomy map along $\mathcal{W}_{\mathcal{F},\mathrm{loc}}^s$ from $\mathcal{W}_{\mathcal{F},\mathrm{loc}}^u(x_0)\cap \Lambda\simeq V_{x_0}$ to $V_{x_1}$. To ease the notation, we also abbreviate $y_1=y_1(y_0):=\mathcal{H}^s_{x_0,x_1}(y_0)$. Note  that, \emph{a priori},  $y_1 \notin \Lambda$. 

\begin{lemma}
	There exists a continuous function $C=C_{x_0}\colon \mathcal{W}_{\mathcal{F},\mathrm{loc}}^s(x_0) \to \R$ such that for any $x_1 \in \mathcal{W}_{\mathcal{F},\mathrm{loc}}^s(x_0)$, it holds   
	\begin{equation}\label{lim holonomy}
	\lim_{\mathcal{W}_{\mathcal{F},\mathrm{loc}}^u(x_0)\cap \Lambda\ni y_0 \to x_0} \frac{d(y_1(y_0),x_1)}{d(y_0,x_0)}=\lim_{\mathcal{W}_{\mathcal{F},\mathrm{loc}}^u(x_0)\cap \Lambda\ni y_0 \to x_0} \frac{\gamma_{y_0}^s(x_1)}{\gamma_{y_0}^s(x_0)}=C(x_1).
	\end{equation}
	Moreover, it holds 
	$$
	\lim_{\mathcal{W}_{\mathcal{F},\mathrm{loc}}^s(x_0)\ni x_1\to x_0} C(x_1)=1. 
	$$
\end{lemma}

\begin{proof}
	According to Remark \ref{remarque reg}, stable holonomy maps 
	are $\mathcal{C}^{1,\beta}$, for some $\beta \in (0,1)$. 
	For any $y_0 \in \mathcal{W}_{\mathcal{F},\mathrm{loc}}^u(x_0)\cap \Lambda$, we let $y_1=y_1(y_0)\in V_{x_1}$ be defined as above. As $\mathcal{H}^s_{x_0,x_1}(x_0)=x_1$ and $\mathcal{H}^s_{x_0,x_1}(y_0)=y_1$, the quotient  in \eqref{lim holonomy} can be written as $\frac{d(\mathcal{H}^s_{x_0,x_1}(y_0),\mathcal{H}^s_{x_0,x_1}(x_0))}{d(y_0,x_0)}$. From the definition of $\gamma_{y_0}^s$, this quantity is also equal to $\frac{\gamma_{y_0}^s(x_1)}{\gamma_{y_0}^s(x_0)}$. Moreover, it has a limit as $y_0 \to x_0$, which we denote by $C(x_1)\in \R$. We thus get a continuous map $C=C_{x_0}\colon \mathcal{W}_{\mathcal{F},\mathrm{loc}}^s(x_0) \to \R$.  Moreover, the holonomy map $\mathcal{H}^s_{x_0,x_1}$ converges to the identity in the $\mathcal{C}^1$ topology as $x_1 \to x_0$, hence $\lim_{x_1 \to x_0} C(x_1)=C(x_0)=1$.  
\end{proof}

For any points $y_0\in \mathcal{W}_{\mathcal{F},\mathrm{loc}}^u(x_0)\cap \Lambda$, $x_1\in \mathcal{W}_{\mathcal{F},\mathrm{loc}}^s(x_0)\cap \Lambda$ close to $x_0$, we also abbreviate $z_1=z_1(y_0):=[y_0,x_1]_{R_i}$. 
Recall that by local product structure, we have $z_1\in \Lambda$. We denote by $\mathscr{Q}(y_0,x_1):=(x_0,x_1,z_1,y_0)\in \Lambda^4$ the associated quadrilateral. 
\begin{lemma}
	For any points $y_0\in \mathcal{W}_{\mathcal{F},\mathrm{loc}}^u(x_0)\cap \Lambda$, $x_1\in \mathcal{W}_{\mathcal{F},\mathrm{loc}}^s(x_0)\cap \Lambda$ close to $x_0$, 
	the area of the quadrilateral $\mathscr{Q}(y_0,x_1)$ is  equal to
	\begin{equation}\label{approximation aire}
	\mathrm{Area}(\mathscr{Q}(y_0,x_1))=(y_0-x_0) (x_1-x_0)[\rho(x_0)+o(1)],
	\end{equation}
	where $\rho$ is the density function of $\mathcal{R}_*(d\alpha)$ introduced above.
	Therefore,  for any points $y_0\in \mathcal{W}_{\mathcal{F},\mathrm{loc}}^u(x_0)\cap \Lambda$,  $x_1,x_2\in \mathcal{W}_{\mathcal{F},\mathrm{loc}}^s(x_0)\cap \Lambda$ close to $x_0$, we have\footnote{We thank Disheng Xu for the idea to use three points $x_0,x_1,x_2$ in the same leaf and consider the ratio of areas to get rid of the ``width'' of quadrilaterals.}
	\begin{equation}\label{quotient des aires}
	\frac{\mathrm{Area}(\mathscr{Q}(y_0,x_2))}{\mathrm{Area}(\mathscr{Q}(y_0,x_1))}=\frac{x_2-x_0}{x_1-x_0}+o(1).
	\end{equation}
\end{lemma}

\begin{proof}
For any $y_0\in \mathcal{W}_{\mathcal{F},\mathrm{loc}}^u(x_0)\cap \Lambda$, $x_1\in \mathcal{W}_{\mathcal{F},\mathrm{loc}}^s(x_0)\cap \Lambda$ close to $x_0$, we have 
$$
\mathrm{Area}(\mathscr{Q}(y_0,x_1))=\int_{x_0}^{x_1}\left(\int_{0}^{\gamma_{y_0}^s(\xi)}  \rho(\xi,\eta) \, d\eta\right) d\xi+o\big((y_1-x_1)^2\big),
$$
where $y_1=y_1(y_0)$.
Here, we use the fact that the unstable lamination $\mathcal{W}_{\mathcal{F}}^u(y_0)$ is $\mathcal{C}^1$, so that the angle between $\mathcal{W}_{\mathcal{F},\mathrm{loc}}^u(x_1)$ and $V_{x_1}$ is going to $0$ as $x_1 \to x_0$, and hence, the area of the missing ``triangle'' bounded by $\mathcal{W}_{\mathcal{F},\mathrm{loc}}^s(y_0)$,  $\mathcal{W}_{\mathcal{F},\mathrm{loc}}^u(x_1)$ and $V_{x_1}$ is a 
$o((y_1-x_1)^2)$, noting that $\rho=O(1)$ on the quadrilateral. Since the argument is a local one, \eqref{lim holonomy} guarantees that $y_1-x_1=O(y_0-x_0)$. In the following, we will always assume that $y_0-x_0 \leq x_1-x_0$, so that $o\big((y_1-x_1)^2\big)=o\big((y_0-x_0)(x_1-x_0)\big)$. Therefore, we obtain
\begin{align*}
\mathrm{Area}(\mathscr{Q}(y_0,x_1))&=\int_{x_0}^{x_1} \gamma_{y_0}^s(\xi)\big( \rho(\xi,0)+O(y_0-x_0) \big) \, d\xi +o\big((y_0-x_0) (x_1-x_0)\big)\\
&=\int_{x_0}^{x_1}(C(\xi)(y_0-x_0) +o(y_0-x_0))\big(\rho(\xi,0)+O(y_0-x_0)\big)\, d\xi \\
&\quad +o\big((y_0-x_0) (x_1-x_0)\big)\\
&=(y_0-x_0) \int_{x_0}^{x_1}\big(C(\xi)\rho(\xi,0) +o(1)\big)\, d\xi +o\big((y_0-x_0) (x_1-x_0)\big)\\
&=(y_0-x_0) (x_1-x_0)[\rho(x_0)+o(1)],
\end{align*}
since $C(\xi)=C(x_0)+o(1)=1+o(1)$, when $\xi \to x_0$. Observe now that \eqref{quotient des aires} follows immediately by taking the quotient.
\end{proof}

For $i=1,2$, let $\Phi_i=(\Phi_i^t)_{t \in \R}$ be a $\mathcal{C}^k$ Axiom A flow defined on a smooth $3$-manifold $M_i$. Let $\Lambda_i$ be a basic set for $\Phi_i$, and let $\alpha_i$ be a smooth contact form adapted to $\Lambda_i$. Assume that there exists a flow conjugacy $\Psi\colon \Lambda_1 \to  \Lambda_2$ between $\Phi_1|_{\Lambda_1}$ and $\Phi_2|_{\Lambda_2}$. For any point $x_0 \in \Lambda_1$, and for $*=s,u$, without loss of generality, because of Lemma \ref{lemma noyau ker} and up to translating along the flow direction, we can assume that $\mathcal{W}_{\Phi_1,\mathrm{loc}}^*(x_0)$, resp.  $\mathcal{W}_{\Phi_2,\mathrm{loc}}^*(\Psi(x_0))$ belongs to some rectangle $R^{(1)}$ of a Markov family for $\Phi_1$, resp. to some rectangle $R^{(2)}$ of a Markov family for $\Phi_2$, so that $\mathcal{W}_{\Phi_1,\mathrm{loc}}^*(x_0)=\mathcal{W}_{R^{(1)}}^*(x_0)$, and $\mathcal{W}_{\Phi_2,\mathrm{loc}}^*(\Psi(x_0))=\mathcal{W}_{R^{(2)}}^*(\Psi(x_0))$.  Moreover, by using some chart as above,   we see $\Psi|_{\mathcal{W}_{\Phi_1,\mathrm{loc}}^*(x_0)}$ as a map from $S_1 \subset \R$ to $S_2 \subset \R$, with $x_0 \simeq 0 \simeq \Psi(x_0)$.  
\begin{prop}\label{differnti}
Assume that the flow conjugacy $\Psi$ is iso-length-spectral. Then, for any point $x_0 \in \Lambda_1$, and for $*=s,u$, the following limit exists:
	$$
	\partial_* \Psi(x_0):=\lim_{\mathcal{W}_{\Phi_1,\mathrm{loc}}^*(x_0)\cap \Lambda\ni x_1 \to x_0} \frac{\Psi(x_1)-\Psi(x_0)}{x_1-x_0}.
	$$
Moreover, the associated map $\partial_* \Psi$ is H\"older continuous on $\Lambda_1$. In other words, for some $\beta \in (0,1)$, the conjugacy $\Psi$ is $\mathcal{C}^{1,\beta}$ along $\mathcal{W}_{\Phi_1,\mathrm{loc}}^s,\mathcal{W}_{\Phi_1,\mathrm{loc}}^u$ in the sense of Whitney. 
\end{prop}

\begin{proof}
	Let us consider the case where $*=s$; the other case is analogous. Fix $x_0 \in \Lambda_1$. 
	Take $y_0\in \mathcal{W}_{\Phi_1,\mathrm{loc}}^u(x_0)\cap \Lambda_1$, $x_1,x_2\in \mathcal{W}_{\Phi_1,\mathrm{loc}}^s(x_0)\cap \Lambda_1$ close to $x_0$. Without loss of generality, we assume that $d(x_0,x_1)\leq d(x_0,x_2)$. By Corollary \ref{egalite aires}, for $i=1,2$, the quadrilaterals $\mathscr{Q}(y_0,x_i)=(x_0,x_i,z_i,y_0)\in \Lambda_1^4$ and $\Psi(\mathscr{Q})(y_0,x_i):=(\Psi(x_0),\Psi(x_i),\Psi(z_i),\Psi(y_0))\in \Lambda_2^4$  have the same area; hence, 
	$$
	\frac{\mathrm{Area}(\mathscr{Q}(y_0,x_2))}{\mathrm{Area}(\mathscr{Q}(y_0,x_1))}=\frac{\mathrm{Area}(\Psi(\mathscr{Q})(y_0,x_2))}{\mathrm{Area}(\Psi(\mathscr{Q})(y_0,x_1))}.
	$$
	We deduce from formula \eqref{quotient des aires} that 
	$$
	\frac{\Psi(x_2)-\Psi(x_0)}{x_2-x_0}=\frac{\Psi(x_1)-\Psi(x_0)}{x_1-x_0}+o\left(\frac{\Psi(x_1)-\Psi(x_0)}{x_2-x_0}\right).
	$$
	For any $x\in \mathcal{W}_{\Phi_1,\mathrm{loc}}^s(x_0)\cap \Lambda_1$ close to $x_0$, we denote 
	$q(x):=\frac{\Psi(x)-\Psi(x_0)}{x-x_0}$. Recall that $d(x_0,x_1)\leq d(x_0,x_2)$; 
	thus, the previous identity can be written as 
	$$
	q(x_1)=q(x_2)+o\big(\max(q(x_1),q(x_2))\big).
	$$
	Now, let us fix a sequence of points $(u_n)_{n \in \N}\in (\mathcal{W}_{\Phi_1,\mathrm{loc}}^s(x_0)\cap \Lambda_1)^{\N}$ going to $x_0$ as $n \to +\infty$. It is easy to see that that $(q(u_n))_{n \in \N}$ is bounded. Consequently, for any $n \geq 0$, $p \geq 0$, the previous identity gives
	$$
	q(u_{n+p})-q(u_n)=o(1). 
	$$
	We deduce that $(q(u_n))_{n \in \N}$ is a Cauchy sequence, hence it converges to some limit $\ell\in \R$. Therefore, for any sequence $(v_n)_{n \in \N}\in (\mathcal{W}_{\Phi_1,\mathrm{loc}}^s(x_0)\cap \Lambda_1)^{\N}$ converging to $x_0$, it holds that $q(v_n) \to \ell$ as $n \to +\infty$. This shows that $\Psi$ is differentiable at $x_0$  along $\mathcal{W}_{\Phi_1,\mathrm{loc}}^s(x_0)$, thus at any point in $\Lambda_1$,  along $\mathcal{W}_{\Phi_1,\mathrm{loc}}^s$.  
	
	In order to show that the map $\partial_s \Psi$ is H\"older continuous on $\Lambda_1$ along $\mathcal{W}_{\Phi_1,\mathrm{loc}}^s$, we argue as follows. Fix $x_0 \in \Lambda_1\cap R^{(1)}$, and let $x_0'\in \mathcal{W}_{R^{(1)}}^s(x_0)\cap \Lambda_1$ be close to $x_0$. Let $(u_n)_{n \in \N}\in (\mathcal{W}_{R^{(1)}}^s(x_0)\cap \Lambda_1)^{\N}$, resp. $(u_n')_{n \in \N}\in (\mathcal{W}_{R^{(1)}}^s(x_0')\cap \Lambda_1)^{\N}$, be a sequence of points in $\Lambda_1$ converging to $x_0$, resp. $x_0'$ along $\mathcal{W}_{R^{(1)}}^s(x_0)=\mathcal{W}_{R^{(1)}}^s(x_0')$.  For any point $y_0\in \mathcal{W}_{R^{(1)}}^u(x_0)\cap \Lambda_1$ close to $x_0$, and for each integer $n \in \N$, we let $\overline{\mathscr{Q}}_n(y_0)=(x_0,u_n,z_n,y_0)\in (\Lambda_1\cap R^{(1)})^4$ and $\overline{\mathscr{Q}}_n'(y_0)=(x_0',u_n',z_n',y_0')\in (\Lambda_1\cap R^{(1)})^4$, where $z_n=[y_0,u_n]_{R^{(1)}}$, $y_0'=[y_0,x_0']_{R^{(1)}}$ and $z_n'=[y_0',u_n']_{R^{(1)}}$. Let $\mathscr{Q}_n(y_0)$, resp. $\mathscr{Q}'_n(y_0)$, be the lift of $\overline{\mathscr{Q}}_n(y_0)$, resp. $\overline{\mathscr{Q}}_n'(y_0)$, as in the proof of Lemma \ref{lemma approx}. We deduce from \eqref{approximation aire} that 
	\begin{align*}
	\mathrm{Area}(\mathscr{Q}_n(y_0))&=(y_0-x_0) (u_n-x_0)[\rho(x_0)+o(1)],\\
	\mathrm{Area}(\mathscr{Q}_n'(y_0))&=(y_0'-x_0') (u_n'-x_0')[\rho(x_0')+o(1)] \\
	&=C_{x_0}(x_0')(y_0-x_0)(u_n'-x_0')[\rho(x_0')+o(1)],
	\end{align*}
	so that 
	\begin{equation*}
	\frac{\mathrm{Area}(\mathscr{Q}_n'(y_0))}{\mathrm{Area}(\mathscr{Q}_n(y_0))}=C_{x_0}(x_0')\frac{u_n'-x_0'}{u_n-x_0}\left(1+O(x_0'-x_0)+o(1)\right).
	\end{equation*} 
	As the images of the quadrilaterals $\mathscr{Q}_n(y_0)$ and $\mathscr{Q}_n'(y_0)$ by $\Psi$ have the same area, we deduce that 
	\begin{align*}
	&C_{\Psi(x_0)}(\Psi(x_0'))\frac{\Psi(u_n')-\Psi(x_0')}{\Psi(u_n)-\Psi(x_0)}\left( 1+O(\Psi(x_0')-\Psi(x_0))+o(1)\right)\\
	&=C_{x_0}(x_0')\frac{u_n'-x_0'}{u_n-x_0}\left( 1+O(x_0'-x_0)+o(1)\right).
	\end{align*}
	Observe that \begin{align*}
	C_{x_0}(x_0')&=1+O(x_0'-x_0), \\ C_{\Psi(x_0)}(\Psi(x_0'))&=1+O(\Psi(x_0')-\Psi(x_0))=1+O(\vert x_0'-x_0\vert^{\beta}),
	\end{align*}
	for some $\beta\in(0,1)$ since $\Psi$ is H\"older continuous. Thus, for $y_0\rightarrow x_0$ we obtain
	\begin{equation*}
	\frac{\Psi(u_n')-\Psi(x_0')}{u_n'-x_0'}=\frac{\Psi(u_n)-\Psi(x_0)}{u_n-x_0}\left( 1+O(|x_0'-x_0|^\beta)\right).
	\end{equation*}
	Letting $n \to +\infty$, we deduce that $|\partial_s \Psi(x_0')-\partial_s \Psi(x_0)|=O(|x_0'-x_0|^\beta)$. Thus, applying Whitney's theorem, we conclude that $\Psi$ is $\mathcal{C}^{1,\beta}$ in the sense of Whitney along $\mathcal{W}_{\Phi_1,\mathrm{loc}}^s$, for $\beta \in(0,1)$. 
\end{proof}

Recall that roughly speaking, Journ\'e's lemma (see \cite{Jour}) says that once a function is regular along the leaves of two transverse foliations, then it is regular globally. It has been generalized by Nicol-T\"or\"ok \cite{Nictamere} in the case of laminations on Cantor sets (see Theorem 1.5 and Remark 1.6 in \cite{Nictamere}). In our case, it reads as follows.

\begin{theorem}[Theorem 1.5 in \cite{Nictamere}]\label{journee}
	Let $\Lambda\subset \R^2$ be a closed, hyperbolic basic set, and for $\beta \in (0,1)$, let  $\mathcal{W}^s$, $\mathcal{W}^u$ be two transverse uniformly $\mathcal{C}^{1,\beta}$ laminations of $\Lambda$. Suppose that $\Theta\colon \Lambda \to \R^2$ is uniformly $\mathcal{C}^{1,\beta}$ in the sense of Whitney when restricted to the leaves of $\mathcal{W}^s$, $\mathcal{W}^u$. Then $\Theta$ is $\mathcal{C}^{1,\beta}$ in the sense of Whitney on $\Lambda$. 
\end{theorem}

From Proposition \ref{differnti} and Theorem \ref{journee}, we then deduce the following
\begin{corollary}\label{cons res}
	 Assume that there exists an iso-length-spectral flow conjugacy $\Psi\colon \Lambda_1 \to  \Lambda_2$ between $\Phi_1|_{\Lambda_1}$ and $\Phi_2|_{\Lambda_2}$. Then $\Psi$ is $\mathcal{C}^{1,\beta}$ in the sense of Whitney on $\Lambda$, for some $\beta \in (0,1)$. 
\end{corollary}

\begin{proof}
	By Proposition \ref{differnti}, we know that $\Psi$ is $\mathcal{C}^{1,\beta}$ in the sense of Whitney along stable/unstable leaves. For $i=1,2$, let us fix a Markov family $\mathcal{R}^{(i)}=\{R_1^{(i)},\dots,R_{m(i)}^{(i)}\}$ with a cross-section $\mathcal{S}^{(i)}$ as given by Theorem \ref{Chernov Markov}. By projecting $\Lambda_1,\Lambda_2$ along flow lines on $\mathcal{S}^{(1)},\mathcal{S}^{(2)}$, and applying Theorem \ref{journee} to the projected sets, we deduce that the map $\widetilde{\Psi}$ induced by $\Psi$ between $\Lambda_1\cap \mathcal{S}^{(1)}$ and $\Lambda_2\cap \mathcal{S}^{(2)}$ is $\mathcal{C}^{1,\beta}$ in the sense of Whitney, for some $\beta \in (0,1)$.  Since the projection along the flow direction is $\mathcal{C}^k$, and since we can describe $\Psi$ in terms of $\widetilde{\Psi}$ and the two projections along $X_1,X_2$,  we conclude that $\Psi$ is $\mathcal{C}^{1,\beta}$ in the sense of Whitney. 
\end{proof}



\subsection{Upgraded regularity of the conjugation}\label{section augm reg}

As previously, let us fix a homeomorphism $\Psi\colon \Lambda_1 \to  \Lambda_2$ 
that is $\mathcal{C}^{1,\beta}$ in Whitney sense, for some $\beta \in (0,1)$, which satisfies 
\begin{equation}\label{equation de conjugazion}
\Psi \circ \Phi_1^t(x)=\Phi_2^{t} \circ \Psi(x),\quad  \text{for all } (x,t) \in \Lambda_1\times \R.
\end{equation}

We will show that the conjugacy map is even more regular:
\begin{prop}\label{prop amel reg}
The conjugacy map $\Psi|_{\Lambda_1}$ is $\mathcal{C}^k$ in  Whitney sense. 
\end{prop}

\begin{proof}
Recall that for $i=1,2$ and $\star=s,u$, there exists $\delta_i^{(\star)}>0$ such that for any $x \in \Lambda_i$, we have 
$$
\delta_i^{(\star)}=\dim_H(\mathcal{W}_{\Phi_i,\mathrm{loc}}^\star(x)\cap \Lambda_i).
$$
As $\Psi$ is $\mathcal{C}^{1,\beta}$, we also have $\delta_1^{(\star)}=\delta_2^{(\star)}=:\delta^{(\star)}$.  

Fix some small $\varepsilon>0$. By Theorem \ref{Chernov Markov}, for $i=1,2$, there exists a proper Markov family $\mathcal{R}^{(i)}=\{R_1^{(i)},\dots,R_{m(i)}^{(i)}\}$ for ${\Phi_i}_{|\Lambda_i}$ of size $\varepsilon$, for some integer $m(i)\geq 1$. Let $\mathcal{S}^{(i)}:=R_1^{(i)}\cup\dots\cup R_{m(i)}^{(i)}$, resp. $\mathcal{F}_i$, be the associated cross-section, resp. Poincar\'e map. We also denote by $\overline{\Lambda}_i:=\Lambda_i \cap \mathcal{S}^{(i)}$ the trace of $\Lambda_i$ on $\mathcal{S}^{(i)}$. The map $\widetilde{\Psi}$ induced by $\Psi$ between $\overline{\Lambda}_1$ and $\overline{\Lambda}_2$ is $\mathcal{C}^{1,\beta}$ in the sense of Whitney. 
Recall that $\dim_H (\overline{\Lambda}_1)=\dim_H (\overline{\Lambda}_2)=\delta^{(s)}+\delta^{(u)}$ (see \cite{Manning} for a reference).


By \cite[Theorem 22.1]{Pesinlecochino}, for $i=1,2$ and $\star=s,u$, there exists a (unique) \emph{equilibrium state}\footnote{\label{formalisme thermo}See for instance \cite{Pesinlecochino} for more details about equilibrium states, potentials, pressure etc.} $\mu_i^\star$ such that for every $x\in \overline{\Lambda}_i$, the conditional measure $m_{i,x}^\star$ of $\mu_i^\star$ on $\mathcal{W}_{\mathcal{F}_i}^\star(x) \cap \overline{\Lambda}_i$ is equivalent to the $\delta^{(\star)}$-Hausdorff measure $H^{\delta^{(\star)}}$. More precisely, $\mu_i^s$ is the equilibrium state for the \emph{potential}\footref{formalisme thermo} $p_i^{(s)}:=\delta^{(s)} \log \|D\mathcal{F}_i|_{E_{\mathcal{F}_i}^s}\|$, and $\mu_i^u$ is the equilibrium state for the potential $p_i^{(u)}:=-\delta^{(u)} \log \|D\mathcal{F}_i|_{E_{\mathcal{F}_i}^u}\|$; besides, the \emph{pressure}\footref{formalisme thermo} $P(p_i^{(\star)})$ vanishes, for $\star=s,u$.  

By \eqref{equation de conjugazion}, for any periodic point $x \in \Lambda_1$ of period $q(x) \geq 1$, the differentials $D\mathcal{F}_1^{q(x)}(x)$ and $D\mathcal{F}_2^{q(x)}(\widetilde{\Psi}(x))$ are conjugate, hence have the same eigenvalues, i.e.,
$$
\sum_{k=0}^{q(x)-1} \left(\log \|D\mathcal{F}_1^k(x)|_{E_{\mathcal{F}_1}^{(\star)}}\|-\log \|D\mathcal{F}_2^k\big(\widetilde{\Psi}(x)\big)|_{E_{\mathcal{F}_2}^{(\star)}}\|\right)=0,\quad  \star=s,u. 
$$ 
By Livsic's Theorem, we deduce that the potentials $\widetilde{\Psi}^* p_2^{(\star)}$ and $p_1^{(\star)}$ are cohomologous, and by \cite[Proposition 4.5]{Bowen75}, we thus have  
$\widetilde{\Psi}^* \mu^{(\star)}_2|_{\overline{\Lambda}_1}=\mu_1^{(\star)}|_{\overline{\Lambda}_1}$. Consequently, $\widetilde{\Psi}^* m_{2,\widetilde{\Psi}(x)}^\star=m_{1,x}^\star$, for $\star=s,u$, and for a.e. $x\in \overline{\Lambda}_1$. 

In the following we deal with the unstable case; the stable one is analogous. To ease the notation, we abbreviate $\delta:=\delta^{(u)}$. For $i=1,2$, and $x_i \in \overline{\Lambda}_i$, the conditional measure $m_{i,x_i}^u$ is equivalent to $H^\delta$, hence we can introduce the density function $\rho_{i,x_i}^u\colon \mathcal{W}_{\mathcal{F}_i}^u(x_i) \to \R^*$, so that $dm_{i,x_i}^u=\rho_{i,x_i}^u dH^\delta$. Recall that the conditional measure $m_{i,x_i}^u$ depends only on the leaf $\mathcal{W}_{\mathcal{F}_i}^u(x_i)$. Our goal in the following paragraph is to show that the function $\rho_{i,x_i}^u(\cdot)/\rho_{i,x_i}^u(x_i)$ is $\mathcal{C}^{k-1}$ in the sense of Whitney. 

As $P(p_i^u)=0$, for any integer $n \geq 0$, and for any $y_i \in \mathcal{W}_{\mathcal{F}_i}^u(x_i)$, we have (see for instance \cite[Section 3.2]{Clim})
\begin{equation}\label{eq densitye}
\frac{d((\mathcal{F}_i^{-n})_* m_{i,x_i}^u)}{dm_{i,\mathcal{F}_i^{-n}(x_i)}^u}(\mathcal{F}_i^{-n}(y_i))=e^{-S_n p_i^u(\mathcal{F}_i^{-n}(y_i))},
\end{equation}
where $S_n p_i^u$ is the $n^{\mathrm{th}}$ Birkhoff sum of $p_i^u$, i.e.,
$$
S_n p_i^u(\mathcal{F}_i^{-n}(y_i)):=\sum_{k=1}^{n} p_i^u(\mathcal{F}^{-k}(y_i))=-\sum_{k=1}^{n} \log \big\|D\mathcal{F}_i^{-1}(\mathcal{F}^{-k}(y_i))|_{E_{\mathcal{F}_i}^u}\big\|^\delta.  
$$ 
In terms of densities,  \eqref{eq densitye} thus yields:
$$
\frac{(\mathcal{F}_i^{-n})_*\rho_{i,x_i}^u}{
	\rho_{i,\mathcal{F}_i^{-n}(x_i)}^u}(\mathcal{F}_i^{-n}(y_i))=\frac{\rho_{i,x_i}^u(y_i)}{
	\rho_{i,\mathcal{F}_i^{-n}(x_i)}^u(\mathcal{F}_i^{-n}(y_i))}=\prod_{k=1}^{n} \big\|D\mathcal{F}_i^{-1}(\mathcal{F}^{-k}(y_i))|_{E_{\mathcal{F}_i}^u}\big\|^\delta. 
$$
Let us consider the ratio of the above quantity and the corresponding one at $x_i$. 
As the distance $d(\mathcal{F}^{-n}(x_i),\mathcal{F}^{-n}(y_i))$ decays exponentially fast with respect to $n$, and assuming that $\mathcal{F}^{-n}(x_i),\mathcal{F}^{-n}(y_i)$ converge to a point $x_i^\infty$ (up to taking subsequences), letting $n \to +\infty$, we obtain
\begin{equation}\label{definit rho i xi}
\rho_i^u(x_i,y_i):=\frac{\rho_{i,x_i}^u(y_i)}{\rho_{i,x_i}^u(x_i)}=\prod_{k=1}^{+\infty} \left(\frac{\|D\mathcal{F}_i^{-1}(\mathcal{F}^{-k}(y_i))|_{E_{\mathcal{F}_i}^u}\|}{\|D\mathcal{F}_i^{-1}(\mathcal{F}^{-k}(x_i))|_{E_{\mathcal{F}_i}^u}\|}\right)^{\delta}.
\end{equation}
In particular, based on that expression, and  
arguing as in \cite[Lemma 4.3]{deLL}, we deduce that the function $\rho_i^u(x_i,\cdot)$ is $\mathcal{C}^{k-1}$ in the sense of Whitney. 

In the rest of this section, we follow the proof of \cite[Lemma 4.5]{deLL}. Fix a point 
$x_1 \in \overline{\Lambda}_1$ and 
let $x_2:=\widetilde{\Psi}(x_1)\in \overline{\Lambda}_2$.  
Since the foliations $\mathcal{W}_{\mathcal{F}_1}^u,\mathcal{W}_{\mathcal{F}_2}^u$ have one dimensional leaves, we can parametrize
patches of the unstable leaves by Riemannian length. Recall that $\widetilde{\Psi}^* m_{2,x_2}^u=m_{1,x_1}^u$; we deduce that for any point $y_1 \in \mathcal{W}_{\mathcal{F}_1}^u(x_1)$, it holds 
(taking charts for $\mathcal{W}_{\mathcal{F}_1}^u(x_1),\mathcal{W}_{\mathcal{F}_2}^u(x_2)$, identifying functions on the leaves and functions of the
coordinates,  and seeing the Whitney extension of $\widetilde{\Psi}|_{\mathcal{W}_{\mathcal{F}_1}^u(x_1)}$ as a map from $\R$ to $\R$):
$$
\int_{x_1}^{y_1} \rho_{1,x_1}^u(s)\,  dH^\delta(s)=\int_{\widetilde{\Psi}(x_1)}^{\widetilde{\Psi}(y_1)} \rho_{2,\widetilde{\Psi}(x_1)}^u(s)\, dH^\delta(s). 
$$
By \eqref{definit rho i xi}, we have 
$$
\rho_{1,x_1}^u(x_1) \int_{x_1}^{y_1} \rho_{1}^u(x_1,s)\,  dH^\delta(s)=\rho_{2,x_2}^u(x_2)\int_{\widetilde{\Psi}(x_1)}^{\widetilde{\Psi}(y_1)} \rho_{2}^u(x_2,s)\, dH^\delta(s). 
$$
For $y_1$ very close to $x_1$, we thus obtain
$$
\rho_{1,x_1}^u(x_1)\int_{x_1}^{y_1}(1+o(1))\, dH^\delta(s)=\rho_{2,x_2}^u(x_2)\int_{\widetilde{\Psi}(x_1)}^{\widetilde{\Psi}(y_1)} (1+o(1))\, dH^\delta(s),
$$
that is
\begin{equation*}
\dfrac{\rho_{1,x_1}^u(x_1)}{\rho_{2,x_2}^u(x_2)}=\dfrac{\int_{\widetilde{\Psi}(x_1)}^{\widetilde{\Psi}(y_1)} \, dH^\delta(s)}{\int_{x_1}^{y_1}\, dH^\delta(s)}+o(1).
\end{equation*}
Consequently,
\begin{align*}
\log\left(\frac{\rho_{1,x_1}^u}{\rho_{2,x_2}^u\circ \widetilde{\Psi}}\right)(x_1)&=\log|\widetilde{\Psi}(y_1)-\widetilde{\Psi}(x_1)|\times \frac{\log\left|\int_{\widetilde{\Psi}(x_1)}^{\widetilde{\Psi}(y_1)}\, dH^\delta\right|}{\log|\widetilde{\Psi}(y_1)-\widetilde{\Psi}(x_1)|}-\\
&\quad -\log|y_1-x_1|\times \frac{\log\left|\int_{x_1}^{y_1}\, dH^\delta\right|}{\log|y_1-x_1|}+o(1).
\end{align*}
When $\mathcal{W}_{\mathcal{F}_1}^u(x_1)\ni y_1 \to x_1$, both $\frac{\log\left|\int_{\widetilde{\Psi}(x_1)}^{\widetilde{\Psi}(y_1)}\, dH^\delta\right|}{\log|\widetilde{\Psi}(y_1)-\widetilde{\Psi}(x_1)|}$ and $\frac{\log\left|\int_{x_1}^{y_1}\, dH^\delta\right|}{\log|y_1-x_1|}$ tend to the dimension  of the measure $H^\delta$, namely, $\delta$. We deduce that
\begin{equation*}
\log\left(\frac{\rho_{1,x_1}^u}{\rho_{2,x_2}^u\circ \widetilde{\Psi}}\right)(x_1)=\delta \log\left(\frac{\widetilde{\Psi}(y_1)-\widetilde{\Psi}(x_1)}{y_1-x_1}\right)+o(1).
\end{equation*}
As $\widetilde{\Psi}|_{\overline{\Lambda}_1}$ is $\mathcal{C}^{1,\beta}$ in the sense of Whitney, letting $\mathcal{W}_{\mathcal{F}_1}^u(x_1)\ni y_1 \to x_1$, we get
\begin{equation*}
\frac{\rho_{1,x_1}^u}{\rho_{2,x_2}^u\circ \widetilde{\Psi}}(x_1)=
(\partial_u\widetilde{\Psi}(x_1))^\delta.
\end{equation*}
In other words, on $\Lambda_1$, the map $\widetilde{\Psi}$ 
satisfies 
\begin{equation}\label{ODE}
\partial_u\widetilde{\Psi}(\cdot)=\left(\dfrac{\rho_{1,(\cdot)}^u(\cdot)}{ \rho_{2,\widetilde{\Psi}(\cdot)}^u\circ \widetilde{\Psi}(\cdot)}\right)^{\frac 1\delta}.
\end{equation}
We have seen that the functions $\rho_{1,(\cdot)}^u,\rho_{2,\widetilde{\Psi}(\cdot)}^u$ are $\mathcal{C}^{k-1}$ in Whitney sense. As $\widetilde{\Psi}$ is $\mathcal{C}^{1,\beta}$ on $\Lambda_1$ along $\mathcal{W}_{\mathcal{F}_1}^u$, the right hand side of \eqref{ODE} is $\mathcal{C}^{1,\beta}$ on $\Lambda_1$ along $\mathcal{W}_{\mathcal{F}_1}^u$.   
We deduce that $\widetilde{\Psi}$ is $\mathcal{C}^{2}$ on $\Lambda_1$ along $\mathcal{W}_{\mathcal{F}_1}^u$ in Whitney sense. By repeating the argument, we conclude that $\widetilde{\Psi}$ is $\mathcal{C}^{k}$ on $\Lambda_1$ along $\mathcal{W}_{\mathcal{F}_1}^u$ in Whitney sense.  
The same arguments applied at stable leaves imply that $\widetilde{\Psi}$ restricted to the leaves of $\mathcal{W}_{\mathcal{F}_1}^s$ is also $\mathcal{C}^k$ in Whitney sense.
By using the version of Journ\'e's Lemma in \cite[Theorem 1.5]{Nictamere} for laminations on hyperbolic sets, and arguing as in the proof of Corollary \ref{cons res}, we conclude that the conjugacy map $\Psi|_{\Lambda_1}$ is $\mathcal{C}^k$ in  Whitney sense, as desired.
\end{proof}

\subsection{Preservation of contact forms: end of the proof of Theorem \ref{thm princ dyn}}\label{sous sect contact}

We have just seen that the flow conjugacy $\Psi$ is $\mathcal{C}^{k}$ in the sense of Whitney on $\Lambda_1$. In this subsection, we show that it implies that $\Psi$ respects the contact structures. See Feldman-Ornstein \cite{FelOrn} for related results in the case of contact Anosov flows on $3$-manifolds.  
\begin{lemma}\label{lemme prer contact}
	We have
	$\Psi^* \alpha_2|_{\Lambda_1}=\alpha_1|_{\Lambda_1}$.
\end{lemma} 

\begin{proof}
	By Lemma \ref{lemma noyau ker}, for $i=1,2$, and for any $x_i \in \Lambda_i$, it holds 
	\begin{equation*}
	E_{\Phi_i}^s(x_i)\oplus E_{\Phi_i}^u(x_i)= \ker \alpha(x_i). 
	\end{equation*}
	Recall that $\Psi$ is a flow conjugacy, i.e., 
	\begin{equation}\label{conj flots}
	\Psi \circ \Phi_1^t(x_1)=\Phi_2^t \circ \Psi(x_1),\quad \forall\, t \in \R,\, x_1\in \Lambda_1.
	\end{equation}
	Therefore, for $*=s,u$, it holds 
	$$
	D\Psi(x_1)E_{\Phi_1}^*(x_1)=E_{\Phi_2}^*(\Psi(x_1)).
	$$
	In particular, 
	$
	\ker \Psi^* \alpha_1(x_1)=\ker \alpha_2(\Psi(x_1))
	$. 
	Moreover, differentiating \eqref{conj flots} with respect to $t$, we obtain $D \Psi (x_1) X_1(x_1)=X_2(\Psi(x_1))$. 
	
	Let us show how this implies the result. We want to show that for any $x\in \Lambda_1$, it holds $\Psi^* \alpha_2(x) =\alpha_1(x)$. For any $v \in T_x M_1$, we decompose it as $v=v^s+v^u + c X_1(x)$, with $v^s\in E_{\Phi_1}^s(x)$, $v^u\in E_{\Phi_1}^u(x)$,  $c \in \R$. We obtain 
	\begin{align*}
	\Psi^* \alpha_2(x)(v)&=\alpha_2(\Psi(x))\big(D\Psi(x)v^s+D\Psi(x)v^u+c D\Psi(x)X_1(x)\big)\\
	&=c\alpha_2(\Psi(x))(D\Psi(x)X_1(x))=c\alpha_2(\Psi(x))(X_2(\Psi(x)))\\
	&=c\imath_{X_2}\alpha_2(\Psi(x))=c=c\imath_{X_1}\alpha_1(x)\\
	&=c\alpha_1(x)(X_1(x))=\alpha_1(x)(v),
	\end{align*}
	which concludes.
\end{proof}

Together with Proposition \ref{upgrading orbit eq} and Corollary \ref{cons res}, this concludes the proof of Theorem \ref{thm princ dyn}.

\section{Smooth conjugacy of billiard maps of hyperbolic billiards}\label{section rig bill} 

 
In the following, we give the proof of Theorem \ref{theorem main billi}. 
Let us consider two billiards $\mathcal{D}_1,\mathcal{D}_2$ with $\mathcal{C}^k$ boundaries, $k\geq 3$, that are iso-length-spectral on two basic sets $\Lambda_1^{\tau_1}$, $\Lambda_2^{\tau_2}$. For $i=1,2$, we denote by $\Phi_i$, resp. $\mathcal{F}_i$, the associated billiard flow, resp. billiard map. Recall that $\Phi_i$ preserves the contact form $\alpha_i:=\lambda_i+dt_i$, where $\lambda_i:=-r_i ds_i$ is the Liouville one-form, and that $\mathcal{F}_i$ preserves the symplectic form $ds_i \wedge dr_i$. 
We let $\Lambda_1,\Lambda_2$ be the respective projections of $\Lambda_1^{\tau_1},\Lambda_2^{\tau_2}$ onto the first two coordinates, i.e.,
\begin{equation}\label{projection lambda i}
\Lambda_i:=\{(s_i,r_i):(s_i,r_i,t_i)\in \Lambda_i^{\tau_i}\text{ for some }t_i \in \R\},\quad i=1,2.
\end{equation} 
By Proposition \ref{upgrading orbit eq}, there exists a flow conjugacy $\widetilde{\Psi}\colon (s_1,r_1,t_1)\mapsto (s_2,r_2,t_2)$ between the billiard flows $\Phi_1|_{\Lambda_1^{\tau_1}}$ and $\Phi_2|_{\Lambda_2^{\tau_2}}$. 
The map $\widetilde{\Psi}$ induces a conjugacy 
$\Psi\colon (s_1,r_1) \mapsto (s_2,r_2)$ between the billiard maps $\mathcal{F}_1|_{\Lambda_1},\mathcal{F}_2|_{\Lambda_2}$. 

\begin{lemma}
	The conjugacy map $\Psi$ is $\mathcal{C}^{k-1}$ in Whitney sense.
\end{lemma}

\begin{proof}
We argue as in Proposition \ref{differnti}, Corollary \ref{cons res} and Proposition \ref{prop amel reg}. The main point is that since the flows $\Phi_1,\Phi_2$ have the same periodic data, quadrilaterals formed by points in $\Lambda_1^{\tau_1}$, $\Lambda_2^{\tau_2}$ in correspondence have the same areas. Let us give more details. Given a point $x_0^{(1)}\in \Lambda_1^{\tau_1}$, we consider a small quadrilateral  $\mathscr{Q}^{(1)}=(x_0^{(1)},x_1^{(1)},x_2^{(1)},x_3^{(1)})\in (\Lambda_1^{\tau_1})^4$ and the associated quadrilateral $\mathscr{Q}^{(2)}:=(x_0^{(2)},x_1^{(2)},x_2^{(2)},x_3^{(2)})\in (\Lambda_2^{\tau_2})^4$, with $x_j^{(2)}:=\widetilde{\Psi}(x_j^{(1)})$, for $j=0,\dots,3$. For $i=1,2$ and $j=0,\dots,3$, we denote by $(s_j^{(i)},r_j^{(i)},t_j^{(i)})$ the coordinates of the point $x_j^{(i)}$ and we let $\bar x_j^{(i)}:=(s_j^{(i)},r_j^{(i)})$ be the projection of $x_j^{(i)}$ on the first two coordinates. Since the flows $\Phi_1|_{\Lambda_1^{\tau_1}}$ and $\Phi_2|_{\Lambda_2^{\tau_2}}$ are conjugated, by Corollary \ref{egalite aires}, we have 
$$
\mathrm{Area}(\mathscr{Q}^{(1)})=\mathrm{Area}(\mathscr{Q}^{(2)}),
$$ 
where for $i=1,2$, $\mathrm{Area}(\mathscr{Q}^{(i)})$ is the area of the region bounded by the points $(\bar x_0^{(i)},\bar x_1^{(i)},\bar x_2^{(i)},\bar x_3^{(i)})\in \Lambda_i^4$. Considering smaller and smaller quadrilaterals, and arguing as in Proposition \ref{differnti}, we deduce that the map $\Psi$ is $\mathcal{C}^{1,\beta}$ in Whitney sense at $x_0^{(1)}$ along stable and unstable leaves, for some $\beta>0$. It follows from Corollary \ref{cons res} that $\Psi$ is $\mathcal{C}^{1,\beta}$ in Whitney sense on $\Lambda_1$. As $\mathcal{F}_1,\mathcal{F}_2$ are $\mathcal{C}^{k-1}$, arguing as in Proposition \ref{prop amel reg}, we can upgrade the regularity and show that $\Psi$ is actually $\mathcal{C}^{k-1}$ in Whitney sense on $\Lambda_1$. 
\end{proof}
Recall that for $i=1,2$, we denote by $\tau_i(s_i,r_i)=h_i(s_i,s_i')>0$ the length of the segment between consecutive bounces $(s_i,r_i)\in \Lambda_i$ and $(s_i',r_i')=\mathcal{F}_i(s_i,r_i)\in \Lambda_i$, so that $\mathcal{F}_i^* \lambda_i-\lambda_i=d\tau_i$. 
By the fact that $\mathcal{D}_1,\mathcal{D}_2$ have the same periodic length data on $\Lambda_1$ and $\Lambda_2$, it follows from Livsic's theorem that the restriction of $\tau_2 \circ \Psi - \tau_1$ to $\Lambda_1$ is a coboundary, i.e., for some continuous function $\chi \colon \Lambda_1 \to \R$, we have 
\begin{equation}\label{def chichi}
\tau_2 \circ \Psi - \tau_1 = \chi\circ \mathcal{F}_1 - \chi\quad \text{on }\Lambda_1. 
\end{equation}
Actually, as $\Psi$ is $\mathcal{C}^{k-1}$ in Whitney sense, by the results of Nicol-T\"or\"ok \cite[Theorem 3.2]{Nictamere}, the function $\chi$ is also $\mathcal{C}^{k-1}$ in Whitney sense. 

\begin{lemma}\label{lemmeliou}
	It holds  
	\begin{equation}\label{eq autre form chi}
	\Psi^* \lambda_2-\lambda_1 = d\chi\quad \text{on }\Lambda_1.
	\end{equation}
	In particular, by differentiating \eqref{eq autre form chi}, it holds $\Psi^* (ds_2\wedge dr_2)=ds_1\wedge dr_1$ on $\Lambda_1$. 
\end{lemma}
\begin{proof}
	Since $\mathcal{F}_i^* \lambda_i-\lambda_i=d\tau_i$, for $i=1,2$, and as $\Psi\circ \mathcal{F}_1|_{\Lambda_1}=\mathcal{F}_2\circ \Psi|_{\Lambda_1}$, we deduce from \eqref{def chichi} that on $\Lambda_1$, 
	\begin{align*}
	\mathcal{F}_1^*\big(\Psi^* \lambda_2-\lambda_1 - d\chi\big)&=\Psi^*(\lambda_2+d\tau_2)-\lambda_1-d\tau_1-\mathcal{F}_1^*d\chi\\
	&=\Psi^* \lambda_2-\lambda_1+d(\tau_2\circ \Psi-\tau_1-\chi\circ \mathcal{F}_1)=\Psi^* \lambda_2-\lambda_1 - d\chi.
	\end{align*}
	Let $\varpi$ be the one-form $(\Psi^* \lambda_2-\lambda_1 - d\chi)|_{\Lambda_1}$. By the above identity, for any $q$-periodic point $p_1\in \Lambda_1$, $q\geq 2$, we have 
	\begin{equation}\label{form varpi}
	\varpi(p_1)=(\mathcal{F}_1^q)^*\varpi(p_1)=\varpi(p_1)\circ D\mathcal{F}_1^q(p_1). 
	\end{equation}
	By the hyperbolicity, we have a splitting $E_{\mathcal{F}_1}^s(p_1)\oplus E_{\mathcal{F}_1}^u(p_1)$ of the tangent space at $p_1$ into stable and unstable spaces. Let us choose a basis $(e^s(p_1),e^u(p_1))\in E_{\mathcal{F}_1}^s(p_1)\times E_{\mathcal{F}_1}^u(p_1)$. We denote by $(d\pi^s_1,d\pi^u_1)$ the dual basis, i.e., for any tangent vector $v$, $d\pi^s_1(v)$, resp. $d\pi^u_1(v)$ denotes the component of $v$ along $e^s(p_1)$, resp. $e^u(p_1)$. 
	In particular, there exist $\alpha^s(p_1),\alpha^u(p_1)\in \R$ such that 
	$$
	\varpi(p_1)=\alpha^s(p_1)\, d\pi^s_1 +\alpha^u(p_1)\, d\pi^u_1. 
	$$
	Letting $0<\mu(p_1)<1<\mu^{-1}(p_1)$ be the eigenvalues of $D\mathcal{F}_1^q(p_1)$, we thus have 
	$$
	\varpi(p_1)\circ D\mathcal{F}_1^{q}(p_1)=\mu(p_1)\alpha^s(p_1)\, d\pi^s_1 +\mu^{-1}(p_1)\alpha^u(p_1)\, d\pi^u_1.
	$$
	We deduce from \eqref{form varpi} that 
	$$
	\big(1-\mu(p_1)\big)\alpha^s(p_1)\, d\pi^s_1 +\big(1-\mu^{-1}(p_1)\big)\alpha^u(p_1)\, d\pi^u_1=0.
	$$
	As $\mu(p_1)\neq 1$, and since $d\pi^s_1,d\pi^u_1$ is a basis of the cotangent space at $p_1$, it follows that $\alpha^s(p_1)=\alpha^u(p_1)=0$. In other words, 
	$\varpi(p_1)=0$, for any periodic point $p_1\in \Lambda_1$. By the continuity of $\varpi$ on $\Lambda_1$, and since periodic points are dense in $\Lambda_1$, we deduce that $\varpi|_{\Lambda_1}=0$, as desired. 
\end{proof}


In order to complete the proof of  Theorem \ref{theorem main billi}, we still need to show \eqref{image temp rev}, which is done in the next subsection.

\subsection{Image of the time-reversal involution by the conjugacy
}\label{section symmetry ies}

We consider the same framework and keep the same notation as above. 
The conjugacy $\Psi$ is not unique, as we may pre-compose, resp. post-compose it with any fixed iterate of $\mathcal{F}_1$, resp. $\mathcal{F}_2$. Yet, in some cases, there is a canonical way to choose the conjugacy in such a way that it preserves the time-reversal symmetry of the billiard dynamics; this is what we discuss in this subsection. Recall that for $i=1,2$, $\mathcal{I}_i \colon (s_i,r_i) \mapsto (s_i,-r_i)$ is the time-reversal involution, so that $\mathcal{F}_i \circ \mathcal{I}_i = \mathcal{I}_i \circ \mathcal{F}_i^{-1}$. In the following, we investigate when it is actually possible to normalize the conjugacy such that it conjugates the time-reversal involutions of $\mathcal{F}_1$ and $\mathcal{F}_2$, i.e., 
\begin{equation}\label{preserv temps sym}
\Psi \circ \mathcal{I}_1=\mathcal{I}_2 \circ \Psi\quad \text{on }\Lambda_1. 
\end{equation}

Let us denote by $\hat{\mathcal{I}}_1:=\Psi^{-1}\circ \mathcal{I}_2 \circ \Psi|_{\Lambda_1}$ the image of $\mathcal{I}_2$ after conjugating by $\Psi$. Clearly, $\hat{\mathcal{I}}_1$ is involutive, and it conjugates $\mathcal{F}_1$ to its inverse $\mathcal{F}_1^{-1}$. In particular, the map $\Gamma:=\hat{\mathcal{I}}_1\circ \mathcal{I}_1$ belongs to the centralizer of the map $\mathcal{F}_1$ on the basic set $\Lambda_1$, i.e.,
$$
\Gamma \circ \mathcal{F}_1=\mathcal{F}_1\circ \Gamma\quad\text{on }\Lambda_1.
$$
The centralizer of Axiom A diffeomorphisms at basic pieces is typically trivial (see \cite{FisherT,RochaVarandas}), hence we expect  $\Gamma$ to be an iterate of $\mathcal{F}_1$. It is actually the case, by \cite[Theorem A]{RochaVarandas}, as long as the map $\Gamma$ fixes the orbits of $\mathcal{F}_1$, i.e., assuming that
\begin{equation}\label{orbit fix}
\forall\, x_1 \in \Lambda_1,\quad \Gamma(x_1)=\mathcal{F}_1^\ell(x_1),\quad \text{for some }\ell=\ell(x_1)\in \Z.
\end{equation}
Actually, we can prove directly:
\begin{lemma}
	If \eqref{orbit fix} holds, then there exists an integer $m \in \Z$ such that 
	\begin{equation}\label{conj inv}
	\mathcal{I}_2 \circ \Psi|_{\Lambda_1}=\Psi\circ \mathcal{I}_1\circ \mathcal{F}_1^m |_{\Lambda_1}. 
	\end{equation}
\end{lemma}

\begin{proof}
	Let $x_1\in \Lambda_1$, and take $\ell=\ell(x_1)\in \Z$ such that \eqref{orbit fix} holds for $x_1$. We have 
	\begin{align*}
	\Gamma(\mathcal{F}_1(x_1))&=\Psi^{-1}\circ \mathcal{I}_2 \circ \Psi\circ \mathcal{I}_1\circ \mathcal{F}_1(x_1)=\Psi^{-1}\circ \mathcal{I}_2 \circ \Psi\circ \mathcal{F}_1^{-1}\circ \mathcal{I}_1(x_1)=\dots=\\
	&=\mathcal{F}_1\circ\Psi^{-1}\circ \mathcal{I}_2 \circ \Psi\circ \mathcal{I}_1(x_1)=\mathcal{F}_1 \circ \Gamma(x_1)=\mathcal{F}_1^{\ell}(\mathcal{F}_1(x_1)),
	\end{align*}
	hence the integer $\ell$ in \eqref{orbit fix} is constant along the orbits. 
	As $\mathcal{F}_1|_{\Lambda_1}$ is transitive, considering $x_1\in \Lambda_1$ with a dense orbit, and by continuity, this finishes the proof. 
\end{proof}	

Besides, in the case where $\mathcal{D}_1,\mathcal{D}_2$ are open dispersing billiards, 
after changing the conjugacy, it is possible to verify \eqref{preserv temps sym}:
\begin{lemma}\label{point de per deux}
	If, furthermore, $\mathcal{F}_1,\mathcal{F}_2$ have a periodic point of period $2$ (in particular, when $\mathcal{D}_1,\mathcal{D}_2\in \mathbf{B}$), then, based on \eqref{conj inv}, we can redefine $\Psi$ so that \eqref{preserv temps sym} holds. 
\end{lemma}

\begin{proof}
	Let us show that the integer $m$ in	\eqref{conj inv} is even. Indeed, let $x_1$ be a periodic point of period $2$ for $\mathcal{F}_1$. Thus, both $\Psi(x_1)$ and $\mathcal{F}_1^m(x_1)$ are $2$-periodic, for $\mathcal{F}_2$ and $\mathcal{F}_1$ respectively. In particular, the point $\Psi(x_1)$, resp. $\mathcal{F}_1^m(x_1)$, is fixed under $\mathcal{I}_2$, resp. $\mathcal{I}_1$. Therefore, by \eqref{conj inv}, $\Psi(x_1)=\Psi(\mathcal{F}_1^m(x_1))$; by the injectivity of $\Psi$, we conclude that $\mathcal{F}_1^m(x_1)=x_1$, hence $m=2\ell$, for some $\ell\in \Z$. Let us consider the map $\widehat{\Psi}:=\Psi \circ \mathcal{F}_1^{-\ell}$. By \eqref{conj inv}, equation \eqref{preserv temps sym} is satisfied for $\widehat{\Psi}$ in place of $\Psi$, as 
	$$
	\mathcal{I}_2 \circ \widehat{\Psi}|_{\Lambda_1}=\Psi\circ \mathcal{I}_1\circ \mathcal{F}_1^{m-\ell}|_{\Lambda_1}=\Psi \circ \mathcal{F}_1^{\ell-m} \circ \mathcal{I}_1|_{\Lambda_1}=\widehat{\Psi}\circ \mathcal{I}_1|_{\Lambda_1}. 
	$$
	Besides, the map $\widehat{\Psi}$ still conjugates $\mathcal{F}_1|_{\Lambda_1}$ to $\mathcal{F}_2|_{\Lambda_2}$, and it is also $\mathcal{C}^{k-1}$ in Whitney sense, which concludes. 
\end{proof}

Alternatively, 
we have:
\begin{lemma}\label{on normailise}
	If there exists $x_1 \in \Lambda_1\cap \{r_1=0\}$ whose orbit is dense in $\Lambda_1$, such that $\mathcal{F}_2^m\circ \Psi(x_1) \in \{r_2=0\}$, for $m \in \Z$, then we can redefine $\Psi$ so that \eqref{preserv temps sym} holds. 
\end{lemma}

\begin{proof}
	Let us assume that there exists a point $x_1 \in \Lambda_1\cap \{r_1=0\}$ whose orbit is dense in $\Lambda_1$, and such that $\mathcal{F}_2^m\circ \Psi(x_1) \in \{r_2=0\}$ for some $m \in \Z$. Let us consider the map $\widehat{\Psi}:=\mathcal{F}_2^m\circ\Psi|_{\Lambda_1}=\Psi \circ \mathcal{F}_1^{m}|_{\Lambda_1}$.  For any integer $\ell \in \Z$, we have 
	\begin{align*}
	\mathcal{I}_2 \circ \widehat{\Psi}(\mathcal{F}_1^\ell(x_1))&= \mathcal{I}_2\circ \mathcal{F}_2^m  \circ \Psi\circ \mathcal{F}_1^{\ell} (x_1)=\dots=\mathcal{F}_2^{-\ell}\circ \mathcal{I}_2  \circ \mathcal{F}_2^m \circ \Psi  (x_1)\\
	&=\mathcal{F}_2^{-\ell}\circ \mathcal{F}_2^m \circ \Psi  (x_1)=\mathcal{F}_2^m \circ \Psi \circ \mathcal{F}_1^{-\ell} (x_1)= \widehat{\Psi} \circ \mathcal{I}_1(\mathcal{F}_1^\ell(x_1)).
	\end{align*}
	In other words, \eqref{preserv temps sym} is satisfied for $\widehat{\Psi}$ in place of $\Psi$ on the orbit of $x_1$; as the latter is dense, and by continuity, it is satisfied everywhere on $\Lambda_1$, which concludes. 
\end{proof}

\begin{remark}
	Note that if there exists a conjugacy map $\Psi$ which satisfies \eqref{preserv temps sym}, then it is unique in the following sense: if $\widehat\Psi$ is another conjugacy map which satisfies \eqref{preserv temps sym} and such that $\Psi^{-1}\circ \widehat\Psi$ fixes $\mathcal{F}_1$-orbits, 
	then $\widehat\Psi=\Psi$. Indeed, in this case, $\Psi^{-1}\circ \widehat\Psi$ commutes with $\mathcal{F}_1$; arguing as above, we see that it is equal to $\mathcal{F}_1^m$, for some $m \in \Z$. 
	Since $\Psi^{-1}\circ \widehat\Psi$ is also in the centralizer of $\mathcal{I}_1$, we deduce that $\mathcal{F}_1^m$ commutes with $\mathcal{I}_1$. But we also have $\mathcal{F}_1^m \circ \mathcal{I}_1=\mathcal{I}_1 \circ \mathcal{F}_1^{-m}$, and hence, $\mathcal{F}_1^{2m}=\mathrm{id}$ on $\Lambda_1$; in other words, each point in $\Lambda_1$ is $2m$-periodic. As $\mathcal{F}_1|_{\Lambda_1}$ is transitive, it is possible only if $m=0$, i.e., $\widehat\Psi=\Psi$ on $\Lambda_1$. 
\end{remark}

Assuming that \eqref{preserv temps sym} holds, we also have the following result.
\begin{lemma}\label{lemme chi impaire}
	The function $\chi$ in \eqref{def chichi} can be chosen such that $\chi \circ \mathcal{I}_1=-\chi$, i.e., 
	\begin{equation}\label{chi symmm}
	\chi(s_1,-r_1)=-\chi(s_1,r_1),\quad \forall\, (s_1,r_1)\in \Lambda_1. 
	\end{equation}
\end{lemma}

\begin{proof}
	For $i=1,2$, we have $\tau_i=\tau_i \circ \mathcal{I}_i\circ \mathcal{F}_i$ (see Figure \ref{Time rev}) and $\mathcal{F}_i \circ \mathcal{I}_i = \mathcal{I}_i \circ \mathcal{F}_i^{-1}$. Thus, by \eqref{preserv temps sym}, we deduce that on $\Lambda_1$, it holds 
	\begin{align*}
	\chi\circ \mathcal{F}_1 - \chi&=\tau_2 \circ \Psi - \tau_1=\tau_2 \circ \mathcal{I}_2 \circ \mathcal{F}_2 \circ \Psi - \tau_1\circ \mathcal{I}_1 \circ \mathcal{F}_1\\
	&=\tau_2 \circ \mathcal{I}_2 \circ \Psi \circ \mathcal{F}_1 - \tau_1\circ \mathcal{I}_1 \circ \mathcal{F}_1=\big(\tau_2 \circ \Psi - \tau_1 \big)\circ \mathcal{I}_1 \circ \mathcal{F}_1\\
	&=\chi\circ \mathcal{F}_1 \circ \mathcal{I}_1 \circ \mathcal{F}_1 - \chi\circ \mathcal{I}_1 \circ \mathcal{F}_1=\chi \circ \mathcal{I}_1 - \chi\circ \mathcal{I}_1\circ \mathcal{F}_1,
	\end{align*}
	hence 
	$$
	(\chi+\chi \circ \mathcal{I}_1)\circ \mathcal{F}_1=\chi+\chi \circ \mathcal{I}_1.
	$$
	Therefore, the function $\chi+\chi \circ \mathcal{I}_1$ on $\Lambda_1$ is $\mathcal{F}_1$-invariant, hence constant, as $\mathcal{F}_1|_{\Lambda_1}$ is transitive and $\chi$ is continuous. Since $\chi$ is defined up to constant (for any $c \in \R$, \eqref{def chichi} also holds for $\chi+c$ in place of $\chi$), we can assume that this constant vanishes, which concludes. 
\end{proof} 

In particular, for any point $x_1=(s_1,0)\in \Lambda_1\cap \{r_1=0\}$, \eqref{chi symmm} gives $\chi(x_1)=0$, while \eqref{eq autre form chi} gives $d\chi(x_1)=0$, as $\Psi(\Lambda_1\cap \{r_1=0\})=\Lambda_2 \cap \{r_2=0\}$. The proof of  Theorem \ref{theorem main billi} is complete.\\

\begin{figure}[h]
	\includegraphics[scale=0.9, trim=0 0.6cm 0 0.5cm, clip]{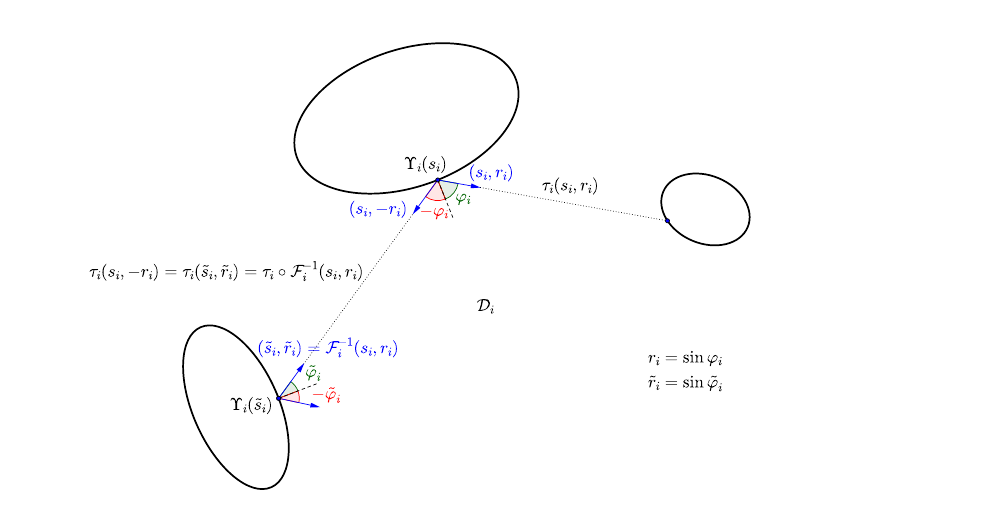}
	\caption{Time-reversal symmetry and generating functions.}\label{Time rev}
\end{figure}

Similarly, the conjugacy $\widetilde{\Psi}$ between the billiard flows $\Phi_1|_{\Lambda_1^{\tau_1}},\Phi_2|_{\Lambda_2^{\tau_2}}$ 
is not unique, as we can pre-, resp. post-compose it with any $\Phi_1^t$, resp. $\Phi_2^t$, $t \in \R$. Yet, there is also a canonical way to choose it, which we now explain. For $i=1,2$, we denote by $\widetilde{\mathcal{I}}_i\colon (x_i,y_i,\omega_i)\mapsto (x_i,y_i,\omega_i+\pi)$ the time reversal involution in $(x_i,y_i,\omega_i)$-coordinates. Let us for instance assume that for $i=1,2$, there exists $X_i\in \Lambda_i^{\tau_i}$ associated to a point on $\partial\mathcal{D}_i$ with a perpendicular bounce, whose orbit is dense, and such that $X_2$, $\widetilde{\Psi}(X_1)$ are in the same orbit. After time-translation, $\widetilde{\Psi}(X_1)=X_2$, and then, 
\begin{equation}\label{conj time inv flot}
\widetilde{\Psi} \circ \widetilde{\mathcal{I}}_1|_{\Lambda_1^{\tau_1}}=\widetilde{\mathcal{I}}_2\circ \widetilde{\Psi}|_{\Lambda_1^{\tau_1}}. 
\end{equation}
To show this, we argue as in Lemma \ref{on normailise}: indeed, 
as $\widetilde{\Psi}(X_1)=X_2$, 
we see that \eqref{conj time inv flot} holds on the orbit of $X_1$, hence everywhere, by the transitivity of $\Phi_1|_{\Lambda_1^{\tau_1}}$.\\ 
Although it is not clear \textit{a priori} that  $\widetilde{\Psi}$ sends points associated to bounces on $\partial\mathcal{D}_1$ to points associated to bounces on $\partial\mathcal{D}_2$, we will show that it is indeed the case when the point on $\partial\mathcal{D}_1$ has a perpendicular bounce. For $i=1,2$, we denote by $\widetilde{\Pi}_i\colon(x_i,y_i,\omega_i)\mapsto (x_i,y_i)$ the projection on the table $\mathcal{D}_i$. 
\begin{lemma}
Assume that \eqref{conj time inv flot} holds. Then, for any $Y_1\in \Lambda_1^{\tau_1}$ associated to a point $\widetilde{\Pi}_1(Y_1)\in \partial\mathcal{D}_1$ with a perpendicular bounce, 
its image $Y_2:=\widetilde{\Psi}(Y_1)\in \Lambda_2^{\tau_2}$ under $\widetilde{\Psi}$ is also associated to a point $\widetilde{\Pi}_2(Y_2)\in \partial\mathcal{D}_2$ with a perpendicular bounce on an obstacle. 
\end{lemma} 

\begin{proof}
	Let $Y_1$, $Y_2:=\widetilde{\Psi}(Y_1)$ be as in the lemma.
	As $Y_1$ has a perpendicular bounce, we have $\widetilde{\mathcal{I}}_1\circ \Phi_1^{-t}(Y_1)=\Phi_1^{t}(Y_1)$, for all $t \in \R$. Since $\widetilde{\Psi}$ conjugates $\Phi_1|_{\Lambda_1^{\tau_1}}$ to $\Phi_2|_{\Lambda_2^{\tau_2}}$, by \eqref{conj time inv flot}, and as 
	 $\widetilde{\Pi}_i\circ\widetilde{\mathcal{I}}_i=\widetilde{\Pi}_i$, $i=1,2$, we deduce that for any $t \in \R$, it holds 
	\begin{align*}
	\widetilde{\Pi}_2\circ \Phi_2^{-t}(Y_2)&=\widetilde{\Pi}_2\circ \widetilde{\mathcal{I}}_2\circ \widetilde{\Psi}\circ\Phi_1^{-t}(Y_1)=\widetilde{\Pi}_2\circ\widetilde{\Psi}\circ\widetilde{\mathcal{I}}_1\circ\Phi_1^{-t}(Y_1)\\
	&=\widetilde{\Pi}_2\circ\widetilde{\Psi}\circ\Phi_1^{t}(Y_1)=\widetilde{\Pi}_2\circ \Phi_2^{t}(Y_2). 
	\end{align*}
	But $\widetilde{\Pi}_2\circ \Phi_2^{-t}(Y_2)=\widetilde{\Pi}_2\circ \Phi_2^{t}(Y_2)$, for all $t \in \R$, if and only if $Y_2$ is associated to a point on $\partial \mathcal{D}_2$ with a perpendicular bounce, which concludes. 
\end{proof}

\subsection{Proof of Corollary \ref{spectral rig cor}}\label{section proof cor e}

As in Corollary \ref{spectral rig cor}, fix $\ell \geq 3$, and let $\mathcal{D}_1, \mathcal{D}_2\in \billiards_{ne}(\ell)$ with $\mathcal{C}^k$ boundaries, for some $k\geq 3$, such that $\mathcal{D}_1,\mathcal{D}_2$ have the same marked length spectrum. Then, according to Theorem \ref{theorem main billi}, the respective billiards maps $\mathcal{F}_1,\mathcal{F}_2$ are conjugated on $\Omega(\mathcal{F}_1),\Omega(\mathcal{F}_2)$ by a map $\Psi\colon \Omega(\mathcal{F}_1) \to \Omega(\mathcal{F}_2)$ that is $\mathcal{C}^{k-1}$ in Whitney sense and such that $\Psi^* (ds_2\wedge dr_2)=ds_1\wedge dr_1$ on $\Omega(\mathcal{F}_1)$. In the following, we let $\Omega_i:=\Omega(\mathcal{F}_i)$, for $i=1,2$. 

Let us recall that $\mathcal{F}_1|_{\Omega_1}$, $\mathcal{F}_2|_{\Omega_2}$ are conjugated to the same subshift of finite type on the alphabet $\mathscr{A}=\{1,\dots,\ell\}$ associated with the transition matrix $(1-\delta_{i,j})_{1 \leq i,j \leq \ell}$, where $\delta_{i,j}=1$, when $i=j$, and $\delta_{i,j}=0$ otherwise. We say that a word $\varsigma=(\varsigma_j)_{j}\in \mathscr{A}^\Z$ is \textit{admissible}, if 
$\varsigma_{j+1}\neq \varsigma_j$, for all $j \in \Z$. We also let $\text{Adm} \subset \cup_{j \geq 2} \mathscr{A}^j$ be the set of all finite words $\sigma=\sigma_1\dots\sigma_j$, $j \geq 2$, such that $\sigma^\infty:=\cdots\sigma\sigma\sigma\cdots \in \text{Adm}_{\infty}$. 
We normalize the conjugacy $\Psi$ by requiring that for each $y_1 \in \Omega_1$, the points $y_1$ and $\Psi(y_1)\in \Omega_2$ are coded by the same admissible word. Symbolically, the actions of $\mathcal{I}_1,\mathcal{I}_2$ amount to switching the symbolic past and future. In particular, by our choice that $\Psi$ preserves the symbolic coding, we have $\Psi \circ \mathcal{I}_1 = \mathcal{I}_2 \circ \Psi$ on $\Omega_1$, where $\mathcal{I}_i \colon (s_i,r_i)\mapsto (s_i,-r_i)$ is the time-reversal involution, for $i=1,2$. 

Due to the equality of the marked length spectra, the respective  generating functions $\tau_1,\tau_2$ of $\mathcal{F}_1,\mathcal{F}_2$ satisfy
\begin{equation*}
\tau_2 \circ \Psi - \tau_1 = \chi\circ \mathcal{F}_1 - \chi\quad \text{on }\Omega_1,
\end{equation*}
for some coboundary $\chi\colon \Omega_1 \to \R$. Arguing as above, we see that $\chi$ is $\mathcal{C}^{k-1}$ in Whitney sense; moreover, properties 
\eqref{liouville im}-\eqref{image temp rev} about $\chi$ follow from the previous results (see Lemma \ref{lemmeliou} and Lemma \ref{lemme chi impaire}), which concludes the proof of Corollary \ref{spectral rig cor}. 

\bibliographystyle{siam}
\bibliography{Bibliography}
\end{document}